\documentclass[12pt]{article}     \usepackage{amsmath}
\usepackage{amsthm}               \usepackage{amsopn}
\usepackage{graphics}
\usepackage{latexsym}             \usepackage{amsfonts}
\usepackage{amssymb}              
\usepackage{amsmath}
\usepackage{enumerate}
\usepackage{array}
\usepackage{graphicx}
\usepackage{amsopn}

\swapnumbers
\theoremstyle{plain}
\newtheorem{theorem}{Theorem}[section]     

\newtheorem{lemma}[theorem]{Lemma}
\newtheorem{corollary}[theorem]{Corollary}

\theoremstyle{definition}
\newtheorem{definition}[theorem]{Definition}
\newtheorem{remark}[theorem]{Remark}

\def\qed{\hfill\rule{1ex}{1ex}\\}
\newenvironment{pf}{\noindent {\bf Proof.}}{\qed}

\title{The Caffarelli Alternative in Measure for the Nondivergence Form Elliptic Obstacle Problem with Principal Coefficients in VMO}
\author{Ivan Blank and Kubrom Teka}

\newcommand{\nc}[2]{ \newcommand{#1}{#2} }

\nc{\avint}{ {- \hspace{-3.5mm} \int} }  
\newcommand{\myavint}[1]{ \int_{#1} \mkern-27mu
                          \rule[.033 in]{.12 in}{.01 in} \ \ \  }
\newcommand{\myavinttwo}[1]{ \int_{#1} \mkern-35mu
                          \rule[.033 in]{.12 in}{.01 in} \ \ \  }
\nc{\R}{{\rm {I \! R}}}  
\nc{\N}{{\rm {I \! N}}}  
\newcommand{\closure}[1]{ \stackrel{\rule{.1 in}{.01 in}}{#1} }
\newcommand{\ohclosure}[1]{ \stackrel{\rule{.15 in}{.01 in}}{#1} }
\newcommand{\dclosure}[1]{ \stackrel{\rule{.2 in}{.01 in}}{#1} }
\newcommand{\tclosure}[1]{ \stackrel{\rule{.3 in}{.01 in}}{#1} }
\newcommand{\qclosure}[1]{ \stackrel{\rule{.4 in}{.01 in}}{#1} }
\newcommand{\pclosure}[1]{ \stackrel{\rule{.5 in}{.01 in}}{#1} }

\newcommand{\chisub}[1]{ {\mathbf{\chi}}_{_{#1}} }

\newcommand{\newsec}[2]{ \section{#1} \label{sec-#2}  
                         \setcounter{equation}{0}     
                         \setcounter{theorem}{0} }    
\newcommand{\refeqn}[1]{ (\!\!~\ref{eq:#1}) } 
\newcommand{\refthm}[1]{ (\!\!~\ref{#1}) }    

\nc{\Holder}{H\"{o}lder\ }

\nc{\ith}{ \ensuremath{\text{i}^{\text{th}}} }
\nc{\jth}{ \ensuremath{\text{j}^{\text{th}}} }
\nc{\kth}{ \ensuremath{\text{k}^{\text{th}}} }
\nc{\curl}{ \nabla \times }
\nc{\Div}{ \nabla \cdot }

\nc{\Ppl}{ \mathcal{M}^{+} }  \nc{\Pmn}{ \mathcal{M}^{-} }

\nc{\smiley}{ $\stackrel{\because}{\smile} \;$ }

\newcommand{\BVP}[4]{  
  \begin{equation}
        \begin{array}{rl}
           #1 & \ \text{in}
               \ \ #4 \vspace{.05in} \\
           #2 & \ \text{on} \ \ \partial #4 \;.
        \end{array}
  \label{eq:#3}
  \end{equation}    }

\newcommand{\BVPb}[3]{   \BVP{#1}{#2}{#3}{ B_{1} }  }

\newcommand{\BVPc}[4]{  
  \begin{equation}
        \begin{array}{rl}
           #1 & \ \text{in}
               \ \ #4 \vspace{.05in} \\
           #2 & \ \text{on} \ \ \partial #4 \;,
        \end{array}
  \label{eq:#3}
  \end{equation}    }

\newcommand{\BVPbc}[3]{   \BVPc{#1}{#2}{#3}{ B_{1} }  }

\newcommand{\BVPcbsn}[3]{   \BVPc{#1}{#2}{#3}{ B_{s_0} }  }

\begin{document}
\numberwithin{equation}{section}
\maketitle

\begin{abstract} \noindent
We study the obstacle problem with an elliptic operator in nondivergence form with principal coefficients in VMO.
We develop all of the basic theory of existence, uniqueness, optimal regularity, and nondegeneracy of the
solutions.  These results, in turn, allow us to begin the study of the regularity of the free boundary, and we
show existence of blowup limits, a basic measure stability result, and a measure-theoretic version of the
Caffarelli alternative proven in \cite{C1}.
\end{abstract}

\newsec{Introduction}{Intro}
We study strong solutions of the obstacle-type problem:
\begin{equation}
           Lw := a^{ij}D_{ij}w = \chisub{ \{ w > 0 \} } \ \ \text{in}
               \ \ B_1 \;,
\label{eq:BasicProb}
\end{equation}
where we look for $w \geq 0.$ (We use Einstein summation notation throughout the paper.)  A strong solution to a
second order partial differential equation is a twice
weakly differentiable function which satisfies the equation almost everywhere.  (See chapter 9 of \cite{GT}.)
We will assume that the matrix $\mathcal{A} = (a^{ij})$ is symmetric and strictly and uniformly elliptic, i.e.
\begin{equation}
  \mathcal{A} \equiv \mathcal{A}^{T} \ \ \text{and} \ \ 0 < \lambda I \leq \mathcal{A} \leq \Lambda I \;,
\label{eq:UniformEllip}
\end{equation}
or, in coordinates:
$$a^{ij} \equiv a^{ji} \ \ \text{and} \ \ 
    0 < \lambda |\xi|^2 \leq a^{ij} \xi_i \xi_j \leq \Lambda |\xi|^2 \ \ \text{for all} \ \xi \in \R^n, \ \xi \ne 0 \;.$$
Our motivations for studying this type of problem are primarily theoretical, although as observed in \cite{MPS} the mathematical
modeling of numerous physical and engineering phenomena can lead to elliptic problems with discontinuous coefficients.
The best specific example of which we are aware for a motivation to study nondivergence form equations with
discontinuous coefficients is the problem of determining the optimal stopping time in probability.

Although we do not want (or even need) any further assumptions for many of our results about the regularity of 
solutions to our obstacle problem, it turns out
that the question of existence of solutions will require us to assume some regularity of our $a^{ij}.$  In fact, there is an
important example due to C. Pucci (found in \cite{T}) which shows that the strict uniform ellipticity of the $a_{ij}$ (i.e.
Equation\refeqn{UniformEllip}\!) is in general not even enough to guarantee the existence of a solution to the corresponding
partial differential equation.  On the other hand, the space of vanishing mean oscillation (VMO) turns out to be a suitable setting
for existence results and a priori estimates as was shown in papers by Chiarenza, Frasca, and Longo
(see \cite{CFL1} and \cite{CFL2}), and it will also turn out to be an appropriate setting for getting some initial results about the
regularity of the free boundary.  (Without using the language of VMO, Caffarelli proved very similar results in \cite{C3}.  To
compare these results, chapter 7 of \cite{CC} and Remark 2.3 of \cite{W} are very helpful.)
It is worth noting that there are results due to Meyers which require a little bit less smoothness of the coefficients if one is content
to work in $L^p$ spaces with $p$ close to $2$ (see \cite{Me}), but in this case, one cannot use the Sobolev embedding to get
continuity of a first derivative except in dimension two.
In any case, we will assume that the $a^{ij}$ belong to VMO when proving existence, and again when we turn to study the
regularity of the free boundary.  It is also worth noting that elliptic and parabolic equations in both divergence and
nondivergence form with coefficients in VMO have received a fair amount of attention lately from Krylov and his coauthors.
(See in particular \cite{K} and \cite{DKL} and the references therein.)

After we prove a key lemma for compactness and a corollary which leads us to nontrivial blowup limits we establish the following
theorem (See Theorem\refthm{CAMS}in this paper) which is modeled after Caffarelli's main results in \cite{C1}:
\begin{theorem}[Caffarelli's Alternative in Measure (Strong Form)]
We assume that $w$ satisfies Equation\refeqn{BasicProb}\!\!, we assume that $a^{ij} \in \text{VMO}$ and satisfy
Equation\refeqn{UniformEllip}\!\!, and finally we assume that $0 \in \partial \{ w > 0 \}.$  Under these hypotheses, for any
$\epsilon \in (0, 1/8),$ there exists an $r_0 \in (0,1),$ and a $\tau \in (0,1)$
such that \newline
\indent if there exists a $t \leq r_0$ such that 
\begin{equation}
   \frac{|\Lambda(w) \cap B_t|}{|B_t|} \geq \epsilon \;,
\label{eq:bigonce}
\end{equation}
\indent then for all $r \leq \tau t$ we have
\begin{equation}
   \frac{|\Lambda(w) \cap B_r|}{|B_r|} \geq \frac{1}{2} - \epsilon \;.
\label{eq:bigallatime}
\end{equation}
The $r_0$ and the $\tau$ depend on $\epsilon$ and on the $a^{ij},$ but they do \textit{not} depend on the
function $w.$
\end{theorem}
\noindent
On the other hand, in the final section of the paper we show that
the VMO assumption is not, by itself, enough to ensure uniqueness of blowup limits at free boundary points.  Indeed, the
final theorem states:
\begin{theorem}[Counter-Example]
There exists $a^{ij} \in \text{VMO}(B_1)$ which satisfies Equation\refeqn{UniformEllip}with $\lambda = 2$ and
$\Lambda = 3,$ there exists a nonnegative solution $w(x)$ to Equation\refeqn{BasicProb}with this matrix $a^{ij},$
and there exists $\{r_n\} \downarrow 0$ such that 
$$\lim_{n \rightarrow \infty} w_{r_{2n + 1}}(x) = \frac{1}{4}((x_n - \beta)_{+})^{2}$$
and
$$\lim_{n \rightarrow \infty} w_{r_{2n}}(x) = \frac{1}{9}\tau(((x_n - \beta)_{+})^{2})$$
where $\tau$ is a rotation, and where the limits have the same convergence as in Theorem\refthm{EBL}(and where as
usual we let $w_{\epsilon}(x) := \epsilon^{-2}w(\epsilon x)$).
\end{theorem}

We turn to an outline of how the paper is laid out.
We start by giving definitions and background. Next, after showing the existence of nontrivial solutions when the $a^{ij}$
belong to VMO, we turn to some of the basic questions in the introductory theory of the obstacle problem.  Namely, we follow
Caffarelli's treatment (see \cite{C4} and \cite{B}), and show nondegeneracy and optimal regularity of the solutions.  Once we
have these tools, we turn to a study of the free boundary regularity and some of our main results.  We start with a technical
lemma which gives us a compactness property of solutions to the obstacle problems that we are studying.  In spite of
numerous hypotheses, we use this lemma twice in fundamental ways.  The first time, we use it to show a measure stability
property of our solutions.  Namely, if our matrix of coefficients $a^{ij}$ is sufficiently close to the identity matrix, then the
solution to our nondivergence form problem will have a zero set which is very close (in Lebesgue measure) to the zero set
of the solution to a new obstacle problem with the same boundary data, but where we replace our general nondivergence
form operator with the Laplacian.  The second time we use our compactness lemma, we prove that if our $a^{ij}$ belong to
VMO and our solution has $0$ in its free boundary, then we can find a sequence of quadratic rescalings: 
$w_{\epsilon}(x) := \epsilon^{-2}w(\epsilon x),$ which converge to a solution of an obstacle problem with constant
coefficients on all of $\R^n.$  These results in turn are then used to first show a measure theoretic version of Caffarelli's
Alternative and second they are used in a construction of a solution with two different blow up limits at the origin.  Our
measure theoretic Caffarelli Alternative is the first theorem that we stated above, and it shows that at every point of the
free boundary the density of the zero set will be well defined and will be either $0$ or $1/2.$
Indeed, just as in the obstacle problem for the Laplacian, the so-called
``regular points'' (which for us are the points where the density of the zero set is $1/2$) form an open subset of the
free boundary.  On the other hand, in the final section, since the averages of VMO functions do not need to converge,
by rescaling along different radii where the limits of the averages converge to different numbers, we show the
existence of the counter-example described in the second theorem stated above.



\newsec{Notation, Conventions, and Background}{NCB}

We will use the following basic notation throughout the paper:
$$
\begin{array}{lll}
\chisub{D} & \ & \text{the chacteristic function of the set} \ D \\
\closure{D} & \ & \text{the closure of the set} \ D \\
\partial D & \ & \text{the boundary of the set} \ D \\
D_{\epsilon} & \ & \text{all} \ x \ \text{such that dist}(x,D) < \epsilon \\ 
x   & \ & (x_1, x_2, \ldots, x_n) \\
x^{\prime} & \ &(x_1, x_2, \ldots, x_{n-1}, 0) \\
B_{r}(x) & \ & \text{the open ball with radius} \ r \ \text{centered at the point} \ x \\
B_{r}  & \ & B_{r}(0) \\
\end{array}
$$
For Sobolev spaces and \Holder spaces, we will follow the conventions found within Gilbarg and Trudinger's book.
In particular for $1 \leq p \leq \infty, \; W^{k,p}(\Omega)$ will denote the Banach space of functions which are
$k$ times weakly differentiable, and whose derivatives of order $k$ and below belong to $L^p(\Omega),$ and
for $0 < \alpha \leq 1, \; C^{k,\alpha}(\ohclosure{\Omega})$ will denote the Banach space of functions which
are $k$ times differentiable on $\ohclosure{\Omega}$ and whose \kth derivatives are uniformly
$\alpha$-\Holder continuous.  (See \cite{GT} for more details.)

When we are studying free boundary regularity, we will frequently assume \\
\begin{equation}
      0 \in \partial \{w > 0\} \;,
\label{eq:OinFB}
\end{equation}
for convenience.
We will make use of the following terminology.  We define:
\begin{equation}
\begin{array}{l}
\displaystyle{\Omega(w) := \{ w > 0 \},} \\
\displaystyle{\Lambda(w) := \{ w = 0\},} \ \ \text{and} \\
\displaystyle{FB(w) := \partial \Omega(w) \cap \partial \Lambda(w) \;.}
\end{array}
\label{eq:BasicSetDefs}
\end{equation}
We will omit the dependence on $w$ when it is clear.  Note also that ``$\Lambda$'' and ``$\Delta$'' each have double duty
and it is necessary to interpret them based on their context.  We use ``$\Lambda$'' for both the zero set and for one of the
constants of ellipticity, and we use ``$\Delta$'' for the both the Laplacian of a function and for the symmetric difference of two
sets in $\R^n.$  (If $A, B \subset \R^n,$ then $A \Delta B := \{ A \setminus B \} \cup \{ B \setminus A \}.$)

We will also be using the BMO and the VMO spaces frequently, and we gather the relevent
definitions here.  (See \cite{MPS}.)  For an integrable function $f$ on a set $S \subset \R^n$ we let
$$f_{S} :=\int_{S} \mkern-21mu
                          \rule[.033 in]{.12 in}{.01 in} \ \ \  f \;.$$
\begin{definition}[BMO and BMO norm]
If $f \in L^1_{loc}(\R^n),$ and
\begin{equation}
    ||f||_{\ast} := \sup_{B} \frac{1}{|B|} \int_{B} |f(x) - f_{B}| \; dx
\label{eq:BMOdef}
\end{equation}
is finite, 
then $f$ is in the space of bounded mean oscillation, or ``$f \in \text{BMO}(\R^n).$''
We will take $||\cdot||_{\ast}$ as our BMO norm.
\end{definition}

\begin{definition}[VMO and VMO-modulus]  
Next, for $f \in \text{BMO},$ we define
\begin{equation}
    \eta_{f}(r) : = \sup_{\rho \leq r, \; y \in \R^n} \ \frac{1}{|B_{\rho}|} \int_{B_{\rho}(y)} |f(x) - f_{_{B_{\rho}(y)}}| \; dx \;,
\label{eq:etadef}
\end{equation}  
and if $\eta_{f}(r) \rightarrow 0$ as $r \rightarrow 0,$ then we say that $f$ belongs to the space of vanishing mean oscillation, or
``$f \in \text{VMO}.$''
$\eta_{f}(r)$ is referred to as the VMO-modulus of the function $f.$
\end{definition}

Since we will need it later, it seems worthwhile to collect some of Caffarelli's results here for the convenience of the reader.
These results can be found in \cite{C1} and \cite{C4}.  We start with a definition which will allow us to measure the
``flatness'' of a set.

\begin{definition}[Minimum Diameter]  \label{MinDiam}
  Given a set $S \in \R^n,$ we define the minimum diameter
  of $S$ (or $m.d.(S)$ ) to be the infimum among the
  distances between pairs of parallel hyperplanes enclosing
  $S.$
\end{definition}

\begin{theorem}[Caffarelli's Alternative]  \label{CaffAlt}
Assume $\gamma$ is a positive number, $w \geq 0,$ and 
$$\Delta w = \gamma\chisub{ \{ w > 0 \} } \ \text{in} \ B_1 \ \ \ \text{and} \ \ \ 0 \in FB(w) \;.$$
There exists a modulus of continuity $\sigma(\rho)$
depending only on $n$ such that either
  \begin{itemize}
    \item[a.]  $0$ is called a \textit{Singular Point} of $FB(w)$ in which case \newline
               $m.d.( \Lambda \cap B_{\rho} ) \leq 
               \rho \sigma(\rho),$  for all $ \rho \leq 1,$ or
    \item[b.]  $0$ is called a \textit{Regular Point} of $FB(w)$ in which case \newline
               there exists a $\rho_0$ such that 
               $m.d.( \Lambda \cap B_{\rho_0} ) \geq
               \rho_0  \sigma(\rho_0),$ and for all $\rho <
               \rho_0, \; m.d.( \Lambda \cap B_{\rho} ) \geq
               C\rho \sigma(\rho_0).$
  \end{itemize}
Furthermore, in the case that $0$ is regular, there exists a
$\rho_1$ such that for any $x \in B_{\rho_1} \cap 
\partial \Omega(w),$ and any $\rho < 2\rho_1,$ we have
  \begin{equation}
  m.d.(\Lambda \cap B_{\rho}(x) ) \geq C\rho \sigma(2\rho_1).
  \label{eq:MDReg}
  \end{equation}
\end{theorem}

\noindent So the set of regular points is an open
subset of the free boundary, and at any singular
point the zero set must become ``cusp-like.''
Examples of solutions with singular points exist and can
be found in \cite{KN}, and in \cite{C4} Caffarelli
has shown that these singular points must lie in a
$C^1$ manifold.  In our setting it will be more suitable to define regular and singular points in a more measure-theoretic fashion,
but for the rest of this section, we mean ``regular'' and ``singular'' in the sense given in Caffarelli's theorem.

\begin{theorem}[Behavior Near a Regular Point]  \label{RegPts}
Suppose that $w$ satisfies the assumptions of Theorem\refthm{CaffAlt}but with the domain $B_1$ replaced
with the domain $B_M$, and suppose $0$ is a regular point of $FB(w).$

Given $\rho > 0,$ there exists
an $\epsilon = \epsilon(\rho)$ and an $M = M(\rho),$ such that if
$m.d.(\Lambda(w) \cap B_1) > 2n\rho,$ then in an appropriate
system of coordinates the following are satisfied for any
$x$ such that $|x'| < \rho / 16$ and $-1 < x_n < 1,$ and for
any unit vector $\tau$ with $\tau_n > 0$ and 
$|\tau'| \leq \rho / 16:$
 \begin{itemize}
   \item[a.] $D_{\tau}w \geq 0 \;.$
   \item[b.] All level surfaces $\{ w = c \}, \; c > 0,$ are
             Lipschitz graphs: $$ x_n = f(x',c) \ \ \text{with}
             \ \  ||f||_{_{Lip}} \leq \frac{C(n)}{\rho} \; .$$
   \item[c.] $D_{e_n}w(x) \geq C(\rho)d(x,\Lambda) \;.$
   \item[d.] For $|\tau'| \leq \rho / 32, \ D_{\tau}w
             \geq C(\rho)d(x,\Lambda) \;.$
 \end{itemize}
\end{theorem}

\begin{theorem}[$C^{1,\alpha}$ Regularity of Regular Points]  \label{C1Al}
Suppose that $w$ satisfies the assumptions of Theorem\refthm{CaffAlt}\!\!, and suppose $0$ is a
regular point of $FB(w).$  There exists a universal modulus of continuity $\sigma(\rho)$ such
that if for one value of $\rho,$ say $\rho_0,$ we have
$$m.d.( \Lambda \cap B_{\rho_0} ) > \rho_0 \sigma(\rho_0),$$
then in a $\rho_0^2$ neighborhood of the origin, the free
boundary is a $C^{1,\alpha}$ surface $x_n = f(x')$ with
\begin{equation}
   ||f||_{_{C^{1,\alpha}}} \leq \frac{C(n)}{\rho_0} \;.
\label{eq:C1AlReg}
\end{equation}
\end{theorem}

\begin{remark}  \label{rmk:GoodScale}
Note that by the last theorem, the $C^{1,\alpha}$
norm of the free boundary will decay in a universal
way at any regular point under the standard quadratic
rescaling if we are allowed to rotate the coordinates.
\end{remark}

Finally, there are two results due to Chiarenza, Frasca, and Longo which will be of fundamental importance throughout this work,
so we will state them here.  These results can be found in \cite{CFL1} and \cite{CFL2}.

\begin{theorem}[Interior Regularity (Taken from Theorem 4.2 of \cite{CFL1})]  \label{CFLIR}
Let $D \subset \R^n$ be open, let $p \in (1, \infty),$
assume $a^{ij}\in \text{VMO}(D)$ and satisfies Equation\refeqn{UniformEllip}\!\!, and let
$$Lu := a^{ij}D_{ij}u $$
for all $x \in D.$ Assume finally that $D^{\prime \prime} \subset \subset D^{\prime} \subset \subset D.$
Then there exists a constant $C$ such that 
\begin{equation}
   ||u||_{W^{2,p}(D^{\prime \prime})} \leq C(||u||_{L^p(D^{\prime})} + ||Lu||_{L^p(D^{\prime})}) \;.
\label{eq:CFL91est}
\end{equation}
The constant $C$ depends on $n, \lambda, \Lambda, p, dist(\partial D^{\prime \prime}, D^{\prime}),$ and quantities which
depend only on the $a^{ij}.$  (In particular, $C$ depends on the VMO-modulus of the $a^{ij}.$)
\end{theorem}

\begin{theorem}[Boundary Regularity (Taken from Theorem 4.2 of \cite{CFL2})]  \label{CFLBR}
Let $p \in (1, \infty)$ and assume that $u \in W^{2,p}(B_1) \cap W^{1,p}_0(B_1).$  Then there exists a constant $C$ such that
\begin{equation}
   ||u||_{W^{2,p}(B_1)} \leq C(||u||_{L^p(B_1)} + ||Lu||_{L^p(B_1)}) \;.
\label{eq:CFL93est}
\end{equation}
The constant $C$ depends on $n, \lambda, \Lambda, p,$ and quantities which
depend only on the $a^{ij}.$  
\end{theorem}

\begin{remark}[$C^{1,1}$ domains are good enough] \label{CFL93C11}
We wrote the last result with balls because we will not apply it on any other type of set, but in \cite{CFL2}, they prove the
result for arbitrary bounded $C^{1,1}$ domains.  Of course for a $C^{1,1}$ domain, the constant $C$ will have dependance on
the regularity of the boundary.
\end{remark}

\begin{corollary}[Boundary Regularity II] \label{CFLBRII}
Let $p \in (1, \infty)$ and assume that $u, \psi \in W^{2,p}(B_1),$ and $u - \psi \in W^{1,p}_0(B_1).$ 
Then there exists a constant $C$ such that
\begin{equation}
   ||u||_{W^{2,p}(B_1)} \leq C(||u||_{L^p(B_1)} + ||Lu||_{L^p(B_1)} + ||\psi||_{W^{2,p}(B_1)}) \;.
\label{eq:CFL93est}
\end{equation}
The constant $C$ depends on $n, \lambda, \Lambda, p,$ and quantities which
depend only on the $a^{ij}.$  
\end{corollary}


\newsec{Existence Theory when $a^{ij} \in$ VMO}{EVMO}
We assume
\begin{equation}
a^{ij} \in \text{VMO.}
\label{eq:VMOass}
\end{equation}
With this assumption coupled with our assumption given in Equation\refeqn{UniformEllip}we hope to show the existence of a
nonnegative solution to Equation\refeqn{BasicProb}with nonnegative continuous Dirichlet data, $\psi,$ given on $\partial B_1.$
In order to ease our exposition later, we will assume that we have extended $\psi$ to be a nonnegative continuous function
onto all of $B_2,$ and for the time being, we will assume that our extended function $\psi$ belongs to $W^{2,p}(B_2)$ for all
$p \in (1, \infty).$

Next, let $\phi(x)$ denote a standard mollifier with support in $B_1,$ and set $\phi_{\epsilon}(x) := \epsilon^{-n} \phi(x/\epsilon).$
In order to approximate the Heaviside function, we let $\Phi_{\epsilon}(t)$ be a function which satisfies
\begin{equation}
\begin{array}{l}
\displaystyle{1. \ \ \ 0 \leq \Phi_{\epsilon}(t) \leq 1, \ \forall t \in \R.} \\
\displaystyle{2. \ \ \  \Phi_{\epsilon}(t) \equiv 0 \ \ \text{if} \ \ t \leq 0.} \\
\displaystyle{3. \ \ \  \Phi_{\epsilon}(t) \equiv 1 \ \ \text{if} \ \ t \geq \epsilon.} \\
\displaystyle{4. \ \ \  \Phi_{\epsilon}(t) \ \ \text{is monotone nondecreasing.}} \\
\displaystyle{5. \ \ \  \Phi_{\epsilon} \in C^{\infty}.}
\end{array}
\label{eq:PhiProps}
\end{equation}
We define $a^{ij}_{\epsilon} := a^{ij} \ast \phi_{\epsilon},$ we define $\psi_{\epsilon} := \psi \ast
\phi_{\epsilon},$ and finally, we let $w_{\epsilon}$ denote the solution to the problem
\BVPb{a^{ij}_{\epsilon}(x) D_{ij} u(x) = \Phi_{\epsilon}(u(x))}{u(x) = \psi_{\epsilon}(x)}{AppProb}
\begin{lemma}[Existence of a Solution to the Semilinear PDE]  \label{EUSSPDE}
The boundary value problem\refeqn{AppProb}has a nonnegative solution in $C^{\infty}(\ohclosure{B_1}).$
\end{lemma}

\begin{pf}
We will show that the solution, $w_{\epsilon},$ exists by a fairly standard method of continuity argument below.  Using the
weak maximum principle it also follows that $w_{\epsilon} \geq 0.$  By Schauder theory it follows that any $C^{2,\alpha}$
solution is automatically $C^{\infty},$ so it will suffice to get a $C^{2,\alpha}$ solution.

We let $S$ be the set of $t \in [0,1]$ such that the following problem is solvable in $C^{2,\alpha}(\ohclosure{B_1}):$
\BVPb{a^{ij}_{\epsilon}(x) D_{ij} u(x) =t \Phi_{\epsilon}(u(x))}{u(x) = \psi_{\epsilon}(x)}{AppProb2}
Equation\refeqn{AppProb2}is solvable for $t=0$ by Schauder Theory.  (See chapter 6 of \cite{GT}.)
Thus, $S$ is nonempty.\\

\noindent
Claim 1: $S$ as a subset of $[0,1]$ is open.\\
Proof.  We define $L^t(u)$ as a map from the Banach space $C^{2,\alpha}(\ohclosure{B_1})$ to the Banach
space $\mathbf{Y}$ which we define as the direct sum 
$C^{\alpha}(\ohclosure{B_1}) \oplus C^{2,\alpha}(\tclosure{\partial B_1}) \;.$  (The new norm can be taken
as the square root of the sums of the squares of the individual norms.)
Our precise definition of $L^t(u)$ is then
$$L^t(u) \; := \; ( \; a^{ij}_{\epsilon} D_{ij}u - t \Phi_{\epsilon}(u) \;, \ u \; ) \;.$$
Doing calculus in Banach space one can verify that, $[DL^{t}(u)]v$ is equal to
$$( \; a^{ij}_{\epsilon} D_{ij}v - t \Phi^{\prime}_{\epsilon}(u) v \; , \ v \; ) \;,$$
and since $\Phi_{\epsilon}(t)$ is monotone increasing and smooth we know that the first component of this expression has the form:
$$a^{ij}_{\epsilon}(x) D_{ij}v(x)-t c(x)v(x) \ \ \text{and} \ \ c(x) \geq 0 \ \forall x \;.$$
By Schauder theory again (see chapter 6 of \cite{GT}) the problem 
\BVPb{a^{ij}_{\epsilon} D_{ij} v - t cv= f}{v = g}{AppProb3ep}
has a unique solution for any pair $(f, \; g) \in \mathbf{Y}$ which satisfies the usual a priori estimates. In other words
$$ [DL^t(u)]^{-1}: \mathbf{Y} \rightarrow C^{ 2,\alpha }(\ohclosure{B_1}) \  \text{is a bounded 1-1 map.} \ $$
Therefore, by the infinite dimensional implicit function theorem in Banach spaces, $S$ is open. \\

\noindent
Claim 2: $S$ is closed.\\
Proof. This step is accomplished using a priori estimates.
We know that $0 \leq t \Phi_{\epsilon}(u(x)) \leq 1$.
So we have $||a^{ij}_{\epsilon}(x) D_{ij}u(x)||_{L^{\infty}(B_1)} \leq 1,$ and so for any $p$ we
have (see Chapter 9 of \cite{GT}),
$$||u||_{ W^{2,p}(B_1) } \leq C \left( 1 + ||\psi_{\epsilon}||_{ C^{0}(\partial B_1) } \right) \leq C.$$  By the Sobolev embedding,
$$||u||_{C^{1,\alpha}(\closure{B_1})} \leq C ||u||_{W^{2,p}(B_1)} \leq C, \  \text{and so} \  
    ||t\Phi_{\epsilon}(u)||_{C^{1,\alpha}(\closure{B_1})} \leq C.$$
Consequently, by Schauder theory again, $u \in C^{3,\alpha}$ and $||u||_{C^{3,\alpha}(\closure{B_1})} \leq C.$
Now by Arzela-Ascoli, if ${t_k} \subset S$ with $t_k \rightarrow t_{\infty} \in [0,1]$, then the corresponding solutions
$u_{t_k}$ must converge uniformly together with their $1^{\text{st}}$ and $2^{\text{nd}}$ derivatives to a $C^{3,\alpha}$
function. This function must then solve the $t_{\infty}$ problem as the left hand sides and right hand sides of the equations
in\refeqn{AppProb2}are converging uniformly. Thus, $S$ is closed, and hence $S$ must be the entire set, $[0,1]$.
\end{pf}

\begin{theorem}[Existence of a Solution to the Free Boundary Problem]   \label{MainExistence}
Assume Equation\refeqn{UniformEllip}holds, assume that $a^{ij} \in \text{VMO},$ and assume that $\psi$ is nonnegative,
continuous, and belongs to $W^{2,p}(B_1)$ for all $p \in (1,\infty).$ Then there exists a nonnegative function
$w \in W^{2,p}(B_1)$ which solves Equation\refeqn{BasicProb}and satisfies
$w - \psi \in W^{2,p}(B_1) \cap W^{1,p}_{0}(B_1)$ for all $p \in (1,\infty).$ In other words, $w$ satisfies:
\BVPb{a^{ij}(x) D_{ij} w(x) = \chisub{\{w > 0\}}(x)}{w(x) = \psi(x)}{AppProb3}
\end{theorem}

\begin{pf}
We let $w_{\epsilon}$ denote the solution to the problem\refeqn{AppProb}\!\!, and we view the $a^{ij}_{\epsilon}$
as elements of VMO, and observe that the VMO-moduli $\eta_{a^{ij}_{\epsilon}}$\!'s (see Equation\refeqn{etadef}\!) are all dominated
by the VMO-modulus of the corresponding $a^{ij}.$  (This fact is alluded to in Remark 2.2 of \cite{CFL1}.)
In fact, we can verify that all of the dependencies on the $a^{ij}$ of the constant within Corollary\refthm{CFLBRII}remain
under control as we send $\epsilon$ to zero.
At this point we can invoke this theorem
to get a uniform bound on the $W^{2,p}(B_1)$ norm of all of the $w_{\epsilon}$'s.  Standard
functional analysis allows us to choose a subsequence $\epsilon_n \downarrow 0,$ an $\alpha < 1,$ and a 
$w \in W^{2,p}(B_1) \cap C^{1,\alpha}(\dclosure{B_1})$ such that $w_{\epsilon_n}$ converges to $w$ strongly in
$C^{1,\alpha}(\dclosure{B_1})$ and weakly in $W^{2,p}(B_1).$  It remains to show that $w$ satisfies Equation\refeqn{BasicProb}\!.

The fact that $w(x) = \psi(x)$ on $\partial B_1$ follows immediately from the uniform convergence of the $w_{\epsilon_n}.$
Next we need to show that the PDE is satisfied almost everywhere.  Everywhere that $w(x) > 0$ it follows easily by the uniform
convergence of the $w_{\epsilon_n}$ that $\Phi_{\epsilon_n}(w_{\epsilon_n}(x))$ converges to $1.$  To show that
$\Phi_{\epsilon_n}(w_{\epsilon_n}(x))$ converges to $0$ almost everywhere on the set $\Lambda := \{ w = 0\}$ we
assume the opposite in order to derive a contradiction.  So, we can assume that there is a new subsequence (still labeled
with $\epsilon_n$ for convenience), such that $$0 < \gamma \leq \int_{\Lambda} \Phi_{\epsilon_n}(w_{\epsilon_n}(x)) \; dx$$
for all $n.$  Using this fact we have:
\begin{alignat*}{1}
0 &< \gamma \\
   & \leq \int_{\Lambda} \Phi_{\epsilon_n}(w_{\epsilon_n}) \; dx \\
   & = \int_{\Lambda} a^{ij}_{\epsilon_n} D_{ij} w_{\epsilon_n} \; dx \\
   & = \int_{\Lambda} (a^{ij}_{\epsilon_n} - a^{ij}) D_{ij} w_{\epsilon_n} \; dx 
      + \int_{\Lambda} a^{ij} (D_{ij} w_{\epsilon_n} - D_{ij} w ) \; dx
      + \int_{\Lambda} a^{ij} D_{ij} w \; dx \\
   & =: I + I\!I + I\!I\!I.
\end{alignat*}
Integral $I$ converges to zero by using \Holder\!'s inequality coupled with the strong convergence of $a^{ij}_{\epsilon}$ to 
$a^{ij}$ in all of the $L^p$ spaces.  Integral $I\!I$ converges to zero by using the weak convergence in $W^{2,p}$ of
$w_{\epsilon_n}$ to $w.$  Finally, integral $I\!I\!I$ is identically zero because the fact that $w \equiv 0$ on $\Lambda$
guarantees that $D^2w$ will be zero almost everywhere on $\Lambda.$
Thus $\Phi_{\epsilon}(w_{\epsilon})$ converges to $\chisub{\{w > 0\}}$ pointwise a.e., and as an immediate corollary
to this statement, $\Phi_{\epsilon}(w_{\epsilon})$ (and therefore also $a^{ij}_{\epsilon} D_{ij} w_{\epsilon}$) converges
weakly to $\chisub{\{w > 0\}}$ in $L^p(B_1)$ for any $1 < p < \infty.$

Again, by Corollary\refthm{CFLBRII}\!\!, we know $D_{ij}w_{\epsilon}$ is uniformly bounded in $L^p, 1 < p< \infty.$
In particular, $$||D_{ij}w_{\epsilon}||_{L^3(B_1)} \leq C \;.$$
Now let $g$ be an arbitrary function in $L^3(B_1),$ then:
\begin{alignat*}{1}
  & \int_{B_1} \left[ (a^{ij}_{\epsilon} D_{ij} w_{\epsilon}) g -(a^{ij}D_{ij}w)g \right] \; dx \\
  & = \int_{B_1} \left[ (a^{ij}_{\epsilon} D_{ij} w_{\epsilon}) g - (a^{ij}D_{ij}w_{\epsilon})g \right] \; dx 
     + \int_{B_1} \left[ (a^{ij} D_{ij} w_{\epsilon}) g - (a^{ij}D_{ij}w)g \right] \; dx \\
  & = I + I\!I.
\end{alignat*}
For any fixed $i,j,$ we can apply the \Holder inequality to see that the function $a^{ij}g$ is an element of $L^{3/2}(B_1),$
and then it follows that $I\!I \rightarrow 0$ from the fact that $D_{ij}w_{\epsilon}$ convereges to $D_{ij}w$ weakly in $L^{3}(B_1).$
On the other hand
$$I \leq ||D_{ij}w_{\epsilon}||_{L^3(B_1)} ||g||_{L^3(B_1)}||a^{ij}_{\epsilon}-a^{ij}||_{L^3(B_1)} \leq 
     C||a^{ij}_{\epsilon}-a^{ij}||_{L^3(B_1)} \rightarrow 0.$$
Hence, $a^{ij}_{\epsilon} D_{ij} w_{\epsilon} \  \text{converges weakly to} \ a^{ij}D_{ij}w$ in $L^3(B_1).$ 
By uniqueness of weak limits, it follows that $a^{ij}D_{ij}w = \chisub{\{w > 0\}}$ a.e.

\end{pf}


\newsec{Basic Results and Comparison Theorems}{BRCT}
In this section we will not need to make any assumptions about the regularity of the $a^{ij}$ besides the most basic ellipticity.
In spite of our weak hypotheses, we will still be able to show all of the basic regularity and nondegeneracy theorems that we
would expect.  The fact that we do not need $a^{ij} \in \text{VMO}$ for any result in this section will allow us to prove a
better measure stability theorem in the next section.  We will make one small regularity assumption, however:  We will assume that
our strong solutions are all continuous, which means that $p$ must be sufficiently large.
\begin{theorem}[Nondegeneracy] \label{NonDeg}
Let $w$ solve\refeqn{BasicProb}\!. If $B_r(x_0) \subset B_1$ and $x_0 \in \; \closure{\Omega}$, then 
\begin{equation}
\sup_{x \in B_r(x_0)} w(x) \geq C r^2 \;,
\label{eq:NonDeg}
\end{equation}
with $C = C(n, \Lambda).$
\end{theorem}

\begin{pf} By continuity we can assume that $x_0 \in \Omega.$  Define $\Omega_r := B_r(x_0)\cap \Omega,$ 
\ $\Gamma_1 := FB \cap B_r(x_0),$ and $\Gamma_2 := \partial B_r(x_0) \cap \Omega.$
Let $\gamma := \frac{1}{2n}||a^{ij}||_{L^{\infty}(\Omega_r)},$ and set
\begin{equation}
    v(x):= w(x) - w(x_0) - \gamma |x-x_0|^2 \;.
\label{eq:Vdef}
\end{equation}
\includegraphics[scale=0.85]{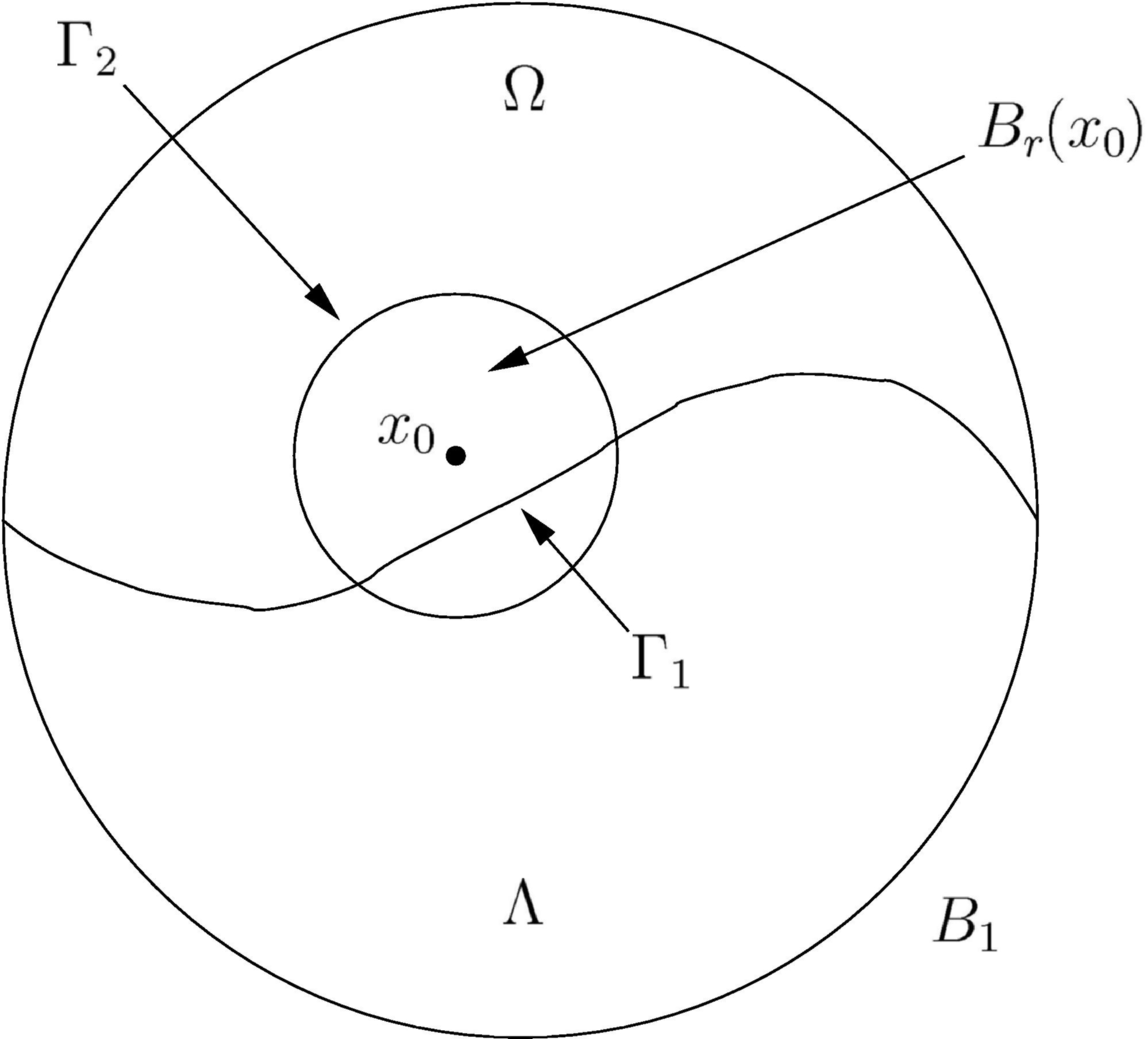} \newline \newline \noindent
Now for $x \in \Omega_r$ we compute:
\begin{alignat*}{1}
Lv &= a^{ij}D_{ij}w - a^{ij}D_{ij}(\gamma|x-x_0|^2) \\
   &= 1 - 2\gamma a^{ij}\delta_{ij} \\
   &= 1 - 2\gamma \sum a^{ii} \\
   &\geq 1 - 2n\gamma||a^{ij}||_{L^{\infty}(\Omega_r)} \\
   &\geq 0 \;.
\end{alignat*}
So now by observing that $v(x_0) = 0,$ by using the weak maximum principle of Aleksandrov
(see Theorem 9.1 of \cite{GT}), and by observing that $v \leq 0$ on $\Gamma_1$ we get
\begin{alignat*}{1}
0 &\leq \sup_{ \Omega_r } v \\
  &\leq \sup_{\partial \Omega_r} v^{+} \\
  &= \sup_{\Gamma_2} v \\
  &= \sup_{\Gamma_2} w - w(x_0)- \gamma r^{2} \\
  &\leq \sup_{\tclosure{B_r(x_0)}} w - w(x_0)- \gamma r^{2} \;.
\end{alignat*}
Now by rearranging terms and observing $w(x_0) \geq 0$ we are done.
\end{pf}

\begin{remark}[Nontrivial Solutions]  \label{nontrivsolns}
As a simple consequence of nondegeneracy, we can take Dirichlet data on $\partial B_1$ which is positive
but small everywhere, to guarantee that we have a solution to our problem which has a nontrivial zero set
and a nontrivial free boundary.  (The origin must be in the zero set in this case.)
\end{remark}

\begin{theorem}[Weak Comparison Principle] \label{WCP}
Let $w_k, \; k = 1,2$ solve\refeqn{BasicProb}\!.  If $w_1 \leq w_2 \leq w_1 + \epsilon$ on $\partial B_1,$
then $w_1 \leq w_2 \leq w_1 + \epsilon$ in $B_1.$
\end{theorem}

\begin{pf} Set $v:= w_1 - w_2,$ and suppose for the sake of obtaining a contradiction that
\begin{equation}
   \max_{x \in B_1} v(x) = v(x_0) = m > 0 \;.
\label{eq:posmaxofv}
\end{equation}
Now we let
\begin{equation}
S^m := \{ x | v(x) = m \} \;.
\label{eq:Smsetdef}
\end{equation}
Since $v$ is a continuous function, there exists a number $\sigma > 0,$ such that 
$v \geq m/2$ on the $\sigma$-neighborhood of $S^m.$  We will denote this set by
$S^m_{\sigma}.$  Now if $S^m_{\sigma}$ extends to the boundary
of the set $B_1,$ then we contradict the fact that $v \leq 0$ on $\partial B_1,$
and thus,
\begin{equation}
S^m_{\sigma} \subset \subset B_1, \ \ \text{and} \ \  v < m \ \text{on} \ \partial S^m_{\sigma} \;.
\label{eq:vlessmbdry}
\end{equation}
Now on this set, since $w_2 \geq 0,$ we must have that $w_1 \geq m/2 > 0.$
Thus, we have
\begin{equation}
Lv = Lw_1 - Lw_2 = 1 - Lw_2 \geq 0 \ \ \ \text{in} \ S^m_{\sigma}.
\label{eq:LisPos}
\end{equation}
By applying the ABP estimate (see \cite{GT} Theorem 9.1) we can conclude that
\begin{equation}
m = \max_{x \in S^m_{\sigma}} v(x) \leq \max_{x \in \partial S^m_{\sigma}} v(x) \;,
\label{eq:resofABP}
\end{equation}
but this equation contradicts the fact that $v < m$ on $\partial S^m_{\sigma}.$

Now we let $w_3$ denote the solution to\refeqn{BasicProb}with boundary data equal to
$w_1 + \epsilon.$  By the first part of the proof, we can conclude that $w_2 \leq w_3$ in
$B_1.$  It remains to show that $w_3 \leq w_1 + \epsilon.$  Suppose not.  Then the function
$u := w_3 - w_1 - \epsilon$ has a positive maximum, $m,$ at a point $x_1.$  Now after observing 
that $w_3(x) > 0$ in a neighborhood of where $u = m$ the proof is identical to the proof of the 
first part.
\end{pf}

\begin{corollary}[Uniqueness] \label{uniqueness}
Any solution to\refeqn{BasicProb}\! with fixed values on $\partial B_1$ is unique.
\end{corollary}

We also can improve our existence theorem easily now to deal with any continuous boundary data:

\begin{corollary}[Improved Existence Theorem] \label{ImpExist}
Assume Equation\refeqn{UniformEllip}holds, assume that $a^{ij} \in \text{VMO},$ and assume that $\psi$ is nonnegative and
continuous. Then there exists a nonnegative function
$w \in W^{2,p}_{loc}(B_1) \cap C^{0}(\ohclosure{B_1})$ $($for all $p \in (1, \infty))$ which satisfies:
\BVPb{a^{ij}(x) D_{ij} w(x) = \chisub{\{w > 0\}}(x)}{w(x) = \psi(x)}{AppProb3}
\end{corollary}

\noindent
The first equality is understood to be in an almost everywhere sense.

\begin{pf} We extend $\psi$ to be a nonnegative continuous function on all of $\ohclosure{B_2}.$  Next we take
$\psi_n \in C^{\infty}(\ohclosure{B_2})$ which are nonnegative and satisfy $$\psi \leq \psi_n \leq \psi + \frac{1}{2^n}\;.$$
By Theorem\refthm{WCP}we get uniform convergence of the corresponding solutions (which we call $w_n$) to a continuous
nonnegative function, $w,$ on all of $\ohclosure{B_1}$  and in fact, we have the estimate
\begin{equation}
w \leq w_n \leq w + \frac{1}{2^n} \;.
\label{eq:wnvsw}
\end{equation}

It is a basic fact from real analysis that on the set $\{ w = 0 \}$ we have $D_{ij}w = 0$ almost
everywhere.  In particular, the equation $$a^{ij}(x) D_{ij} w(x) = 0$$ holds almost everywhere on this
set automatically.  Now because $w$ is continuous, the set where it is positive is an open set, and so we
can suppose that $B_{2r}(x_0) \subset \{ w > 0 \}.$  It follows from
Theorem\refthm{CFLIR}that in $B_{r}(x_0)$ we will have $L^p$ convergence of the second derivatives
$D_{ij}w_n \rightarrow D_{ij}w.$  After taking a subsequence we have convergence almost everywhere, and
so we must have $a^{ij}D_{ij}w = 1$ almost everywhere in $\{ w > 0 \}.$

Finally, in order to get $w \in W^{2,p}_{loc}(B_1)$ we simply observe that Theorem\refthm{CFLIR}will
imply that for any $D \subset \subset B_1,$ and for any $p \in (1, \infty),$ we know that $w_n$ are all
bounded in $W^{2,p}(D)$ and so we can get a subsequence to converge weakly in $W^{2,p}(D)$ to a
function which must therefore be our function $w.$
\end{pf}

\begin{lemma}[Bound on $B_{1/2}$] \label{bound}
If $w \geq 0 $ satisfies Equations\!\refeqn{BasicProb}\!and\!\refeqn{OinFB}\!, then
$ w(x) \leq C(n,\lambda, \Lambda)$ in $\tclosure{B_{1/2}}.$
\end{lemma}
\begin{pf}
Write $w:= w_1 + w_2$, where

\BVPbc{Lw_1 = \chisub{ \{ w > 0 \} }}{w_1 \equiv 0}{Bw1def}
and
\BVPb{Lw_2 = 0}{w_2 = w}{Bw2def}
Then $w_1 \leq 0$ in $B_1$ by the maximum principle.  On the other hand, by the ABP estimate
(Theorem 9.1 \cite{GT}) we have, $w_1|_{_{B_1}} \geq -C.$ Also, by Corollary 9.25 \cite{GT}, along with the fact that
$w_1(0)  + w_2(0) = w(0) = 0$ we have:
$$w_2|_{_{B_{\frac{1}{2}}}} \leq \ \sup_{B_{\frac{1}{2}}} w_2 \leq \ C\inf_{B_{\frac{1}{2}}}w_2 \leq \ Cw_2(0)= \ -Cw_1(0) \leq C.$$
Hence $w|_{_{B_{\frac{1}{2}}}} \leq C.$
\end{pf}

\begin{theorem}[Parabolic Bound] \label{parabound}
If $w \geq 0 $ satisfies\!\refeqn{BasicProb}\!and\!\refeqn{OinFB}\!\!, then
$$w(x) \leq 4C(n,\lambda, \Lambda)|x|^2 \ \text{in} \ B_{1/2},$$
where the constant $C(n,\lambda, \Lambda)$ is the exact same constant
as the constant appearing in the statement of the previous lemma.
\end{theorem}
\begin{pf} Suppose not. Then, $w(\tilde{x}) > 4C(n,\lambda, \Lambda)|\tilde{x}|^2$ for some 
$\tilde{x} \in {B_{1/2}},$ and since $0 \in FB,$ we must have $\tilde{x} \neq 0.$
Now set $\lambda := 2|\tilde{x}|$ so that if $x := \lambda^{-1} \tilde{x},$
then we have $x \in \partial B_{1/2}.$  Define:
\begin{equation}
 w_{\lambda}(x):= {\lambda}^{-2}w(\lambda x) \;.
\label{eq:scaling}
\end{equation}
Clearly $w_{\lambda}$ satisfies\!\refeqn{BasicProb}\!and\!\refeqn{OinFB}\!in $B_1.$ So by the lemma above:
\begin{equation}
w_{\lambda}(x) \leq C(n,\lambda, \Lambda) \ \ \text{in} \ \tclosure{B_{1/2}}.
\label{eq:boundonscaling}
\end{equation}
On the other hand, $${\lambda}^2w_{\lambda}(x) =w(\lambda x)= w(\tilde{x}) 
         > 4C(n,\lambda, \Lambda)|\tilde{x}|^2= C(n,\lambda, \Lambda){\lambda}^2,$$
and so
\begin{equation}
w_{\lambda}(x) >  C(n,\lambda, \Lambda) \;,
\label{contradtobound}
\end{equation}
which contradicts Equation\!\refeqn{boundonscaling}\!.
\end{pf}


\newsec{Compactness and Measure Stability}{ARMS}
So far, except to prove our existence theorem, we have not made any assumptions about our
$a^{ij}$ beyond ellipticity.  In order to prove regularity theorems about the free boundary in the
next section, we will need to assume once again that the $a^{ij} \in \text{VMO}.$  In this section,
on the other hand, we will not assume
$a^{ij} \in \text{VMO},$ but many of our hypotheses anticipate that assumption later.
Now we need a technical compactness lemma which we will need to prove measure stability in this section and which we
will use again when we prove the existence of blow up limits in the next section.
%
\begin{lemma}[Basic Compactness Lemma]  \label{BCL}
Fix $\gamma > 0, \; 1 < p < \infty$ and let $\sigma(r)$ be a modulus of continuity.  Assume that we are given the following:
\begin{enumerate}
   \item $0 < \lambda I \leq a^{ij,k}(x) \leq \Lambda I,$ for a.e. $x.$ 
   \item $w_k \geq 0$ with $L^kw_k := a^{ij,k}D_{ij}w_k = \chisub{\{w_k > 0 \}} \ \text{in} \  B_1.$ 
   \item $0 \in \text{FB}_k, \; \text{so} \; w_k(0) = | \nabla w_k(0)| = 0.$ 
   \item $||w_k||_{W^{2,p}(B_{1})} \leq \gamma.$
   \item $A^{ij}$ is a symmetric, constant matrix with $0 < \lambda I \leq A^{ij} \leq \Lambda I,$ and such that
$||a^{ij,k} - A ^{ij}||_{L^1(B_1)} < \sigma(1/k).$ 
\end{enumerate}
Then for any $\alpha < 1$ and any $p < \infty$ there exists a function
$w_{\infty} \in W^{2,p}(B_{1}) \cap C^{1,\alpha}(\dclosure{B_{1}})$ and a subsequence of the $w_k$
(which we will still refer to as $w_k$ for ease of notation) such that
\begin{itemize}
   \item[A.] $w_k \rightarrow w_{\infty}$ strongly in $C^{1,\alpha}(\dclosure{B_{1}}),$
   \item[B.] $w_k \rightharpoonup w_{\infty}$ weakly in $W^{2,p}(\dclosure{B_{1}}),$ and
   \item[C.] $A^{ij}D_{ij} w_{\infty} = \chisub{\{w_{\infty} > 0 \}} \ \ \ \text{and} \ \ \
            0 \in FB_{\infty} := \partial \{w_{\infty} = 0\} \cap B_{1} \;.$
\end{itemize}
\end{lemma}
\begin{pf}
By using the fourth assumption, we immediately have both A and B from elementary functional analysis and the Sobolev
Embedding Theorem.  We also note that our assumptions of uniform ellipticity actually force a uniform $L^{\infty}$ bound
on all of the $a^{ij,k}$ and the $A^{ij}.$  That bound, together with the fact that
$a^{ij,k}\stackrel{L^1}{\rightarrow} A^{ij},$ allow us to interpolate to any strong convergence in $L^q.$  In other words,
by using the fact that $$||u||_{L^q} \leq ||u||_{L^1}^{(1/q)} \cdot ||u||_{L^{\infty}}^{(1 - (1/q))}$$
(see for example Equation (7.9) in \cite{GT}), we can assert that for $q < \infty$ we have
$a^{ij,k}\stackrel{L^q}{\rightarrow} A^{ij}.$  From this equation it follows that for any $\varphi \in L^{\infty}$ we have
\begin{equation}
a^{ij,k}\varphi \stackrel{L^q}{\rightarrow} A^{ij}\varphi \;.
\label{eq:lqconvofaijkphi}
\end{equation}

\begin{remark}[A Possible Improvement]  \label{API}
It seems to be worth observing that if we were to assume that the $a^{ij,k} \in \text{VMO}$ and we removed the
assumption of uniform ellipticity, then we could still use the theorem of John and Nirenberg to get strong convergence
in $L^q.$  On the other hand, too many of the other proofs rely on the uniform ellipticity of the elliptic operators for us
to tackle this issue in the current paper.
\end{remark}
Returning to the proof and letting $S$ be an arbitrary subset of $B_{1}$ we have
\begin{alignat*}{1}
\int_S {a^{ij,k}D_{ij}w_k} &= \int_S{(a^{ij,k}D_{ij}w_k - A^{ij}D_{ij}w_k + A^{ij}D_{ij}w_k)}\\
&= \int_S{(a^{ij,k} - A^{ij})D_{ij}w_k} + \int_S{ (A^{ij}D_{ij}w_k - A^{ij}D_{ij}w_{\infty} + A^{ij}D_{ij}w_{\infty})}\\
&= \int_S{(a^{ij,k} - A^{ij})D_{ij}w_k} + \int_S{ A^{ij}(D_{ij}w_k - D_{ij}w_{\infty})}+ \int_S{A^{ij}D_{ij}w_{\infty}} \\
&= I + I\!I + \int_S{A^{ij}D_{ij}w_{\infty}}.
\end{alignat*}
The integral $I$ now goes to zero by combining Equation\refeqn{lqconvofaijkphi}with the fourth assumption and then
using \Holder\!'s inequality.  The integral $I\!I$ goes to zero by using B.  Thus we can conclude
\begin{equation}
\int_S a^{ij,k}D_{ij}w_k \rightarrow \int_S A^{ij}D_{ij}w_{\infty}
\label{eq:gottheleft}
\end{equation}
for arbitrary $S \subset B_{1},$ and in particular, the convergence is also pointwise a.e.

Now we claim: $\chisub{\{w_k > 0 \}} \rightarrow \chisub{\{w_{\infty} > 0 \}} \  \text{a.e in} \  B_1.$  Since we
already know that $a^{ij,k}D_{ij}w_k \rightarrow A^{ij}D_{ij}w_{\infty}$ a.e. and since 
$a^{ij,k}D_{ij}w_k = \chisub{\{w_k > 0 \}}$  a.e.,
if we show our claim, then it will immediately imply that 
\begin{equation}
A^{ij}D_{ij}w_{\infty} =  \chisub{\{w_{\infty} > 0 \}} \ \ \text{a.e.}
\label{eq:dontoverlookthis}
\end{equation}
Since we obviously have $||\chisub{\{w_k > 0 \}}||_{L^p(B_1)} \leq C$ for all $p \in (1,\infty],$ elementary functional analysis
implies the existence of a function $g \in L^{\infty}(B_1)$ with $0 \leq g \leq 1$ such that
\begin{equation}
\chisub{\{w_k > 0 \}} \rightharpoonup g \ \text{in}\  L^p , 1 < p < \infty \;.
\label{eq:weakconv}
\end{equation}
Now, wherever we had $w_{\infty} > 0,$ it is immediate that $\chisub{\{w_k > 0 \}}$ converges pointwise (and therefore weakly)
to $1$ by the uniform convergence of $w_k$ to $w_{\infty}.$  In particular, $g \equiv 1$ on $\{ w_{\infty} > 0 \}.$

Next we show that $g \equiv 0$ \ in $ \{w_{\infty} =0\}^\circ.$  So, we suppose that $\qclosure{B_r(x_0)}\  \subset \{w_{\infty} = 0\},$
and we claim that $w_k \equiv 0$ in $ \pclosure{B_{r/2}(x_0)}$ for $k$ sufficiently large.  Suppose not. Then applying
Theorem\refthm{NonDeg}(the nondegeneracy result) to the offending $w_k$'s, we have a sequence
$\{x_k\} \subset \qclosure{B_r(x_0)}$ such that $w_k(x_k) \geq C({r/2})^2.$ 
On the other hand, $w_{\infty}(x_k) \equiv 0$ (since $\qclosure{B_r(x_0)} \; \subset \{w_{\infty} = 0\}$) and this fact
contradicts the uniform convergence of $w_k$ to $w_{\infty}.$

At this point we have $g(x) \equiv 1$ for $x \in \{ w_{\infty} > 0 \},$ and $g(x) \equiv 0$ for $x \in \{w_{\infty}=0\}^\circ $ and so
$g$ agrees with $\chisub{\{w_{\infty} > 0\}}$ on this set.  By the arguments above, the convergence to $g$ is actually
pointwise on this set.  Now we finish this proof by showing that
the set $\mathcal{P} := \{x: |\chisub{\{w_{\infty} > 0\}} - g| \neq 0\}$ has measure zero, and it follows from the preceding
arguments that $\partial \{w_{\infty} = 0\} \subset \mathcal{P} \;.$

We will show that $\mathcal{P}$ has measure zero by showing that it has no Lebesgue points.  To this end, let
$x_0 \in \mathcal{P}$ and let $r$ be positive, but small enough so that $B_r(x_0) \subset B_{1}.$
Define $W_{\infty}(x):= r^{-2}w_{\infty}(x_0 + rx)$ and define $W_j(x):= r^{-2}w_j(x_0 + rx),$ and observe that
all of the convergence we had for $w_j$ to $w_{\infty}$ carries over to convergence for $W_j$ to $W_{\infty},$ except that
now everything is happening on $B_1.$

From our change of coordinates, it follows that $0 \in \partial \{ W_{\infty} = 0 \}$ and since $W_{\infty} \geq 0,$ there exists a
sequence $\{x_k\} \rightarrow 0$ such that $W_{\infty}(x_k) > 0$ for all $k.$  Now fix $k$ so that $x_k \in B_{1/8},$ and
then take $J$ sufficiently large to ensure that if $i,j \geq J$ then the following hold:
\begin{equation}
    ||W_{j} - W_{\infty}||_{L^{\infty}(B_1)} \leq \frac{W_{\infty}(x_k)}{2} \;, \ \ \ \text{and} \ \ \ 
    ||W_{i} - W_{j}||_{L^{\infty}(B_1)} \leq \frac{\tilde{C}}{10}
\label{eq:UseCauchyI}
\end{equation}
where $\tilde{C}$ is a constant which will be determined from the nondegeneracy theorem, and which will be named
momentarily.  The existence of such a $J$ follows from the fact that $W_j$ converges to $W_{\infty}$ in
$C^{1,\alpha}(\dclosure{B_1}).$

We use the first estimate in Equation\refeqn{UseCauchyI}to guarantee that $W_{J}(x_k) > 0.$
We apply Theorem\refthm{NonDeg}to $W_J$ at $x_k$ to guarantee the existence of a point $\tilde{x} \in B_{1/2}$
such that
\begin{equation}
W_{J}(\tilde{x}) \geq C (3/8)^2 \;.
\label{eq:bigxtil}
\end{equation}
Putting this equation together with the second convergence statement in Equation\refeqn{UseCauchyI}and letting
$\tilde{C}$ be defined by the constant on the right hand side of Equation\refeqn{bigxtil}we see that for $i \geq J$
we have:
\begin{equation}
   W_{i}(\tilde{x}) \geq \frac{9\tilde{C}}{10} \;.
\label{eq:bigxtilalliI}
\end{equation}
Since all of the $W_i$'s satisfy a uniform $C^{1,\alpha}$ estimate, there exists an $\tilde{r} > 0$ such that
$W_{i}(y) \geq \tilde{C}/2$ for all $y \in B_{\tilde{r}}(\tilde{x})$ once $i \geq J.$  From this fact we conclude that
$B_{\tilde{r}}(\tilde{x}) \subset \{ W_{\infty} > 0 \}.$

Scaling back to the original functions, we conclude that within $B_{r}(x_0)$ is a ball, $B,$ with radius equal to $r\tilde{r}$
such that $B \subset \{ w_{\infty} > 0 \} \subset \mathcal{P}^c \;.$  Since this type of statement will be true for any $r$ sufficiently
small, we are guaranteed that $x_0$ is \textit{not} a Lebesgue point of $\mathcal{P}.$  Since $x_0$ was arbitrary, we can
conclude that $\mathcal{P}$ has measure zero.

Finally we observe that the nondegeneracy theorem implies immediately that $0$ remains in the free boundary in the
limit.
\end{pf}

\begin{corollary}[Hausdorff Dimension of the Free Boundary]   \label{HDFB}
If $w$ satisfies Equation\refeqn{BasicProb}with coefficients which satisfy Equation\refeqn{UniformEllip}\!\!,
then the free boundary
$$\partial \{ w = 0 \} \cap \partial \{ w > 0 \}$$
is strongly porous and therefore has Hausdorff dimension strictly less than $n.$  In particular, its Lebesgue
n-dimensional measure is zero.
\end{corollary}

For the definition of strongly porous and other basic facts about porosity we refer the reader to Mattila's
book and the references within it.  (See \cite{Ma}.)  Since the proof of this corollary is a repetition of the proof
above that $\mathcal{P}$ has measure zero, we omit it.

\begin{theorem}[Basic Measure Stability Result]  \label{BMSR}
Suppose $w \in W^{2,p}(B_1)$ satisfies\refeqn{BasicProb}and\refeqn{OinFB}\!, assume
$\epsilon > 0, \ p,q > n,$ and $||a^{ij} - \delta ^{ij}||_{L^q(B_1)} < \epsilon,$ and
let $u$ denote the solution to \BVPb{\Delta u = \chisub{\{u > 0\}}}{u \equiv w}{udef}  Then there is a modulus of
continuity $\sigma$ whose definition depends only on $\lambda, \Lambda, p, q, n,$ and $||w||_{W^{2,p}(B_{1})}$
such that
\begin{equation}
| \{ \Lambda (u) \; \Delta \; \Lambda (w) \} \cap B_{1}| \leq \sigma (\epsilon).
\label{eq:symmetricdiff}
\end{equation}
\end{theorem}
\noindent
(Here we use ``$\Delta$'' first to denote the Laplacian and next to denote the symmetric difference between two sets:
$A \Delta B = \{A \setminus B\} \cup \{B \setminus A\}.$)\\
\begin{pf} Let $\gamma := ||w||_{W^{2,p}(B_{1})},$ and suppose the theorem is false.
Then there exist $w_k , u_k \ \text{and} \ a^{ij,k} $ such that:
\begin{enumerate}
   \item $L^kw_k = a^{ij,k}D_{ij}w_k = \chisub{\{w_k > 0 \}} \ \text{in} \  B_1.$ 
   \item $0 \in \text{FB}, w_k(0) = | \nabla w_k(0)| = 0.$ 
   \item $0 < \lambda I \leq a^{ij,k} \leq \Lambda I.$ 
   \item $||a^{ij,k} - \delta ^{ij}||_{L^q(B_1)} < \frac{1}{2^k}.$ 
   \item $\Delta u_k = \chisub{\{ u_k > 0 \}} \ \text{in} \ B_{1} \ \text{and}\  u_k \equiv w_k \ \text{on}\  \partial B_{1}.$ 
   \item $||w_k||_{W^{2,p}(B_{1})} \leq \gamma.$
\end{enumerate}
But,
\begin{equation}
     | \Lambda (u_k) \Delta \Lambda (w_k) \cap B_{1}| \geq \eta > 0 \ \text{for some fixed}\  \eta > 0.
\label{eq:symdiff}
\end{equation}
We invoke the last lemma to guarantee the existence of a function $w_{\infty}$ which satisfies:
\begin{equation}
\Delta w_{\infty} = \chisub{\{w_{\infty} > 0 \}}  \ \ \text{a.e.}
\label{eq:grabbingfruit}
\end{equation}
and has $0 \in FB_{\infty}.$  The last lemma also guarantees that we have $w_k$ converging to $w_{\infty}$
strongly in $C^{1,\alpha}$ and weakly in $W^{2,p}.$

Now we will use Equation\refeqn{symdiff}to get to a contradiction.  We have
\begin{alignat*}{1}
   0 &< \eta \\
      &\leq |\Lambda (u_k) \Delta \Lambda (w_k) \cap B_{1}| \\
      &= ||\chisub{\{u_k > 0\}} - \chisub{\{w_k > 0\}}||_{L^1(B_{1})} \\
      &\leq ||\chisub{\{u_k > 0\}} - \chisub{\{w_{\infty} > 0\}}||_{L^1(B_{1})} + 
            ||\chisub{\{w_{\infty} > 0\}} - \chisub{\{w_k > 0\}}||_{L^1(B_{1})} \\
      &=: I + I\!I \;.
\end{alignat*}
Since \cite{C2} guarantees that the boundary of the set $\{ w_{\infty} \equiv 0 \}$ has finite
$(n-1)$-Hausdorff dimensional measure, it must have zero $n$-dimensional Lebesgue measure.
Thus, we can use uniform convergence to deal with the positivity set, and we can use nondegeneracy to
deal with the interior of the zero set, and thus we can conclude that $\chisub{\{w_k > 0\}}$ converges to
$\chisub{\{w_{\infty} > 0\}}$ almost everywhere.  Then we can apply Lebesgue's Dominated
Convergence Theorem to see that $I\!I \rightarrow 0.$

In order to show $I \rightarrow 0,$ we first note that $w_{\infty}$ and $u_k$ satisfy the same obstacle
problem within $B_{1},$ and on $\partial B_{1}$ we know that $u_k$ equals $w_k$ which in turn converges in
$C^{1,\alpha}$ to $w_{\infty}.$  Now by a well-known comparison principle for the obstacle problem
(see for example, Theorem 2.7(a) of \cite{B}) we know that
\begin{equation}
||u_k - w_{\infty}||_{L^{\infty}(B_{1})} \leq ||u_k - w_{\infty}||_{L^{\infty}(\partial B_{1})} \;.
\label{eq:firstwmp}
\end{equation}
At this point we can quote Corollary 4 of \cite{C2} to finally conclude that $I \rightarrow 0$ and thereby obtain our
contradiction.
\end{pf}

\begin{corollary}[Uniform Stability]  \label{UnifStab}
Suppose $w \in W^{2,p}(B_1)$ satisfies\refeqn{BasicProb}and\refeqn{OinFB}\!, assume
$\epsilon > 0, \ p,q > n,$ and $||a^{ij} - \delta ^{ij}||_{L^q(B_1)} < \epsilon,$ and
let $u$ denote the solution to \BVPb{\Delta u = \chisub{\{u > 0\}}}{u \equiv w}{udef}  Then there is a modulus of
continuity $\sigma$ whose definition depends only on $\lambda, \Lambda, p, q, n,$ and $||w||_{W^{2,p}(B_{1})}$
such that
\begin{equation}
||u - w||_{L^{\infty}(B_1)} \leq \sigma (\epsilon).
\label{eq:LinfDiff}
\end{equation}
\end{corollary}

\begin{pf} By Calderon-Zygmund theory, if the Laplacian of $u - w$ is small in $L^r,$ then $u - w$ will be small in
$W^{2,r}.$  (See Corollary 9.10 in \cite{GT}.)  If $r > n/2,$ then smallness in $W^{2,r}$ guarantees smallness
in $L^{\infty}$ by applying the Sobolev Embedding Theorem.
\begin{alignat*}{1}
\Delta (u - w) &= \chisub{\{u > 0\}} - (\delta^{ij} - a^{ij} + a^{ij}) D_{ij}w \\
                      &= (\chisub{\{u > 0\}} - \chisub{\{w > 0\}}) + (a^{ij} - \delta^{ij})D_{ij} w \\
                      &=: I + I\!I .
\end{alignat*}
The fact that $I$ is small in any $L^r$ follows from the fact that it is bounded between $-1$ and $1$ (to get control of its
$L^\infty$ norm), and is as small as we like in $L^1$ by Theorem\refthm{BMSR}\!.  In order to guarantee that $I\!I$ is
small in $L^{r}$ for some $r > n/2,$ we first observe that $D_{ij}w$ is bounded in $L^p$ for some $p > n,$ and
$||a^{ij} - \delta ^{ij}||_{L^q(B_1)}$ is as small as we like by our hypotheses.  Now we simply apply \Holder\!'s inequality.
\end{pf}


\newsec{Regularity of the Free Boundary}{RFB}
We turn now to a study of the free boundary in the case where the $a^{ij} \in \text{VMO}.$  We will show 
the existence of blowup limits and it will follow from this
result together with the measure stability result from the previous section, that a form of the Caffarelli Alternative will
hold in a suitable measure theoretic sense.

\begin{theorem}[Existence of Blowup Limits]  \label{EBL}
Assume $w$ satisfies\refeqn{BasicProb}and\refeqn{OinFB}\!\!, assume $a^{ij}$ satisfies\refeqn{UniformEllip}\!and
belongs to VMO, and define the rescaling
$$w_{\epsilon}(x) := \epsilon^{-2} w(\epsilon x) .$$
Then for any sequence $\{ \epsilon_n \} \downarrow 0,$ there exists a subsequence (which we will still call $\{ \epsilon_n \}$
to simplify notation) and a symmetric matrix $A = (A^{ij})$ with
$$0 < \lambda I \leq A \leq \Lambda I$$ such that for all $1 \leq i,j \leq n$ we have
\begin{equation}
\myavinttwo{B_{\epsilon_n}} a^{ij}(x) \; dx \rightarrow A^{ij} \;,
\label{eq:weakpointlim}
\end{equation}
and on any compact set, 
$w_{\epsilon_n}(x)$ converges strongly in $C^{1,\alpha}$ and weakly in $W^{2,p}$ to a function
$w_{\infty} \in W^{2,p}_{loc}(\R^n),$ which satisfies:
\begin{equation}
   A^{ij} D_{ij}w_{\infty} = \chisub{ \{ w_{\infty} > 0 \} } \ \ \text{on} \ \R^n ,
\label{eq:globalbu}
\end{equation}
and has $0$ in its free boundary.
\end{theorem}

\begin{remark}[Nonuniqueness of Blowup Limits]  \label{NBL}
Notice that the theorem does not claim that the blowup limit is unique.  In fact, it is relatively easy to produce nonuniqueness,
and we will give such an example in the next section.
\end{remark}

\begin{pf}
Because the matrix $a^{ij}(x)$ satisfies $0 < \lambda I \leq a^{ij}(x) \leq \Lambda I$ for all $x,$ it is clear that if we define 
the matrix
\begin{equation}
   A^{ij}_r := \myavint{B_r} a^{ij}(x) \; dx,
\label{eq:AijrDef}
\end{equation}
then this matrix must also
satisfy the same inequality.   Of course, since all of the entries are bounded, we can take a subsequence of the radii $\epsilon_{n}$
such that each scalar $A^{ij}_{\epsilon_{n}}$ converges to a real number $A^{ij}.$  With this subsequence, we already know
that we satisfy Equation\refeqn{weakpointlim}\!\!, but because $a^{ij}(x) \in$ VMO, we also know:
$$\myavinttwo{B_{\epsilon_n}} \left| a^{ij}(x) - A^{ij}_{\epsilon_{n}} \right| \; dx \leq \eta(\epsilon_{n}) \rightarrow 0 \;,$$
where $\eta$ is just taken to be the maximum of all of the VMO-moduli for each of the $a^{ij}$'s,
and by the triangle inequality this leads to
\begin{equation}
\myavinttwo{B_{\epsilon_n}} \left| a^{ij}(x) - A^{ij} \right| \; dx \rightarrow 0 \;.
\label{eq:strongpointlim}
\end{equation}

Now we observe that if $a^{ij,n}(x) := a^{ij}(\epsilon_n x)$ then the rescaled function $w_{n} := w_{\epsilon_n}$ satisfies the equation:
\begin{equation}
    a^{ij,n}(x) D_{ij} w_{n}(x) = \chisub{ \{ w_{n} > 0 \} }(x) \;,
\label{eq:rescver}
\end{equation}
and
\begin{equation}
\myavint{B_{1}} \left| a^{ij,n}(x) - A^{ij} \right| \; dx \leq \eta(\epsilon_{n}) \rightarrow 0 \;.
\label{eq:happyaijn}
\end{equation}
By combining Theorem\refthm{parabound}with Corollary\refthm{CFLBRII}we get the existence of a constant $\gamma < \infty$
so that $||w_n||_{W^{2,p}(B_{1})} \leq \gamma$ for all $n.$  At this point we satisfy all of the hypotheses of Lemma\refthm{BCL}\!\!,
and applying that lemma gives us exactly what we need.
\end{pf}

\begin{theorem}[Caffarelli's Alternative in Measure (Weak Form)]  \label{CAMW}
Under the assumptions of the previous theorem, the limit
\begin{equation}
    \lim_{r \downarrow 0} \frac{ |\Lambda(w) \cap B_r| }{ |B_r| }
\label{eq:densitystatement}
\end{equation}
exists and must be equal to either $0$ or $1/2.$
\end{theorem}

\begin{pf}
We will suppose that
\begin{equation}
    \limsup_{r \downarrow 0} \frac{ |\Lambda(w) \cap B_r| }{ |B_r| } >  0
\label{eq:densitystatement2}
\end{equation}
and show that in this case the limit exists and is equal to 1/2.
It follows immediately from this assumption that there exists
a sequence $\{ \epsilon_n \} \downarrow 0$ such that (for some $\delta > 0$) we have
\begin{equation}
     \frac{ |\Lambda(w_{\epsilon_n}) \cap B_1| }{ |B_1| } >  \delta
\label{eq:densitystatement3}
\end{equation}
for all $n.$  (Here again we use the quadratic rescaling: $w_{s}(x) := s^{-2}w(sx),$
and we will even shorten ``$w_{\epsilon_n}$'' to ``$w_n$'' henceforth.)
We can now apply the last theorem to extract a subsequence (still called ``$\epsilon_n$''), and
to guarantee the existence of a
symmetric positive definite matrix $A^{ij}$ with all of its eigenvalues in $[\lambda, \Lambda],$
and a $w_{\infty} \in W^{2,p}_{loc}(\R^n),$ such that if $a^{ij,n}(x) := a^{ij}(\epsilon_n x),$ then
\begin{equation}
\myavint{B_{1}} \left| a^{ij,n}(x) - A^{ij} \right| \; dx \rightarrow 0 \;.
\label{eq:strongpointlim2}
\end{equation}
and
\begin{equation}
   A^{ij} D_{ij}w_{\infty} = \chisub{ \{ w_{\infty} > 0 \} } \ \ \text{on} \ \R^n ,
\label{eq:globalbu2}
\end{equation}
and $0$ is in $FB(w_{\infty}).$  Furthermore, we will have $w_n$ converging to $w_{\infty}$ in both
$W^{2,p}$ and $C^{1,\alpha}$ for all $p$ and $\alpha$ on every compact set.

Now we make an orthogonal change of coordinates on $\R^n$ to diagonalize the matrix $A^{ij},$ and then we dilate the individual
coordinates by strictly positive amounts depending only on $\lambda$ and $\Lambda$ so that in the new coordinate system we have
$A^{ij} = \delta^{ij}.$  Now of course, there are new functions, and the constants may change by positive factors that we can
control, but all of the equations above remain qualitatively unchanged, and we will abuse notation (in a manner similar to the fact
that we have not bothered to rename the subsequences), by continuing to refer to our new functions in the new coordinate system as
$w_{n}$ and $w_{\infty},$ and by continuing to refer to the ``new'' $a^{ij,n}$ as $a^{ij,n},$ etc.

Now we let $u_n$ denote the solution to \BVPb{\Delta u_n = \chisub{\{u_n > 0\}}}{u_n \equiv w_n}{undef}
Using Equations\refeqn{densitystatement3}and\refeqn{strongpointlim2}and applying our measure stability result to $u_n$
and $w_n$ we can make $|\Lambda(u_n) \Delta \Lambda(w_n)|$ as small as we like for $n$ sufficiently large.  In particular,
we now have:
\begin{equation}
     \frac{ |\Lambda(u_{n}) \cap B_{1}| }{ |B_{1}| } >  \frac{\delta}{2} \;.
\label{eq:densitystatement4}
\end{equation}

Since $w_n$ converges uniformly to $w_{\infty}$ on every compact set, it follows that $u_n$ converges uniformly to $w_{\infty}$
on $\partial B_{1},$ and now we start arguing exactly as in the last paragraph of the proof of our measure stability theorem.
In particular, Equation\refeqn{firstwmp}holds, and Corollary 4 of \cite{C2} then gives us
\begin{equation}
     \frac{ |\Lambda(w_{\infty}) \cap B_{1}| }{ |B_{1}| } >  \frac{\delta}{2} \;.
\label{eq:densitystatement5}
\end{equation}
Of course now we can invoke the $C^{1,\alpha}$ regularity at regular points (see Theorem\refthm{C1Al}\!\!) to guarantee that
$w_{\infty}$ is $C^{1,\alpha}$ at the origin, and this in turn implies that
\begin{equation}
   \lim_{r \downarrow 0} \frac{ |\Lambda(w_{\infty}) \cap B_{r}|}{|B_{r}|} = \frac{1}{2} \;.
\label{eq:winfgoodatzero}
\end{equation}

Now it remains to do two things.  First we need to pass this result from $w_{\infty}$ back to our subsequence of radii for $w,$ but
second we will then need to show that we get the same limit along any sequence of radii converging to zero.  The first step is a
consequence of combining our measure stability theorem with Corollary 4 of \cite{C2} again.  Indeed, for any $r > 0,$
\begin{equation}
  \lim_{n \rightarrow \infty} 
  \left( \frac{| \; \Lambda(w_{n}) \cap B_r \; |}{| \; B_r \; |}
     - \frac{| \; \Lambda(w_{\infty}) \cap B_r \; |}{| \; B_r \; |}
  \right) = 0 \;.
  \label{eq:Blim}
\end{equation}
On the other hand, by our rescaling, this equation becomes
\begin{equation}
  \lim_{n \rightarrow \infty}
  \left( \frac{| \; \Lambda(w) \cap B_{(r \epsilon_n)} \; |}{| \; B_{(r \epsilon_n)} \; |}
         - \frac{| \; \Lambda(w_{\infty}) \cap B_r \; |}{| \; B_r \; |}
  \right) = 0 \;,
  \label{eq:BetterBlim}
\end{equation}
which we can combine with Equation\refeqn{winfgoodatzero}to ensure that
\begin{equation}
  \lim_{n \rightarrow \infty} \frac{| \; \Lambda(w) \cap B_{(r \epsilon_n)} \; |}{| \; B_{(r \epsilon_n)} \; |}
         = \frac{1}{2} \;.
  \label{eq:BestBlim}
\end{equation}

Finally, we wish to be able to replace ``$r \epsilon_n$'' with ``$r$'' in Equation\refeqn{BestBlim}\!\!.  Suppose that
we have a different sequence of radii converging to zero (which we can call $s_k$) such that
\begin{equation}
  \lim_{k \rightarrow \infty} \frac{| \; \Lambda(w) \cap B_{s_k} \; |}{| \; B_{s_k} \; |}
         \ne \frac{1}{2} \;.
  \label{eq:BestBlim2}
\end{equation}
At this point we are led to a contradiction in one of two ways.  If the limit above does not equal zero (including the case
where it simply does not exist), then we can simply
use Theorem\refthm{EBL}combined with Theorem\refthm{BMSR}to get convergence to a global solution with properties
which contradict the Caffarelli Alternative (Theorem\refthm{CaffAlt}\!\!).  On the other hand, if the
limit does equal zero, then we use the continuity of the function:
$$g(r) := \frac{| \; \Lambda(w) \cap B_{r} \; |}{| \; B_{r} \; |}$$
to get an interlacing sequence of radii which we can call $\tilde{s}_k$ and which converge to zero such that
$g(s_k) \equiv 1/4,$ and then we proceed as in the first case.
\end{pf}

\begin{definition}[Regular and Singular Free Boundary Points]  \label{RSFBP}
A free boundary point where $\Lambda$ has density equal to $0$ is referred to as \textit{singular}, and 
a free boundary point where the density of $\Lambda$ is $1/2$ is referred to as \textit{regular}. 
\end{definition}

The theorem above gives us the alternative, but we do not have any kind of uniformity to our convergence.  Caffarelli
stated his original theorem in a much more quantitative (and therefore useful) way, and so now we will state and prove a similar
stronger version.  We will need the stronger version in order to show openness and stability under perturbation of the regular
points of the free boundary.

\begin{theorem}[Caffarelli's Alternative in Measure (Strong Form)]  \label{CAMS}
Under the assumptions of the previous theorem, for any $\epsilon \in (0, 1/8),$  there exists an $r_0 \in (0,1),$ and a $\tau \in (0,1)$
such that \newline
\indent if there exists a $t \leq r_0$ such that 
\begin{equation}
   \frac{|\Lambda(w) \cap B_t|}{|B_t|} \geq \epsilon \;,
\label{eq:bigonce}
\end{equation}
\indent then for all $r \leq \tau t$ we have
\begin{equation}
   \frac{|\Lambda(w) \cap B_r|}{|B_r|} \geq \frac{1}{2} - \epsilon \;,
\label{eq:bigallatime}
\end{equation}
and in particular, $0$ is a regular point according to our definition.
The $r_0$ and the $\tau$ depend on $\epsilon$ and on the $a^{ij},$ but they do \textit{not} depend on the
function $w.$
\end{theorem}

\begin{remark}[Another version]  \label{AnoVer}
The theorem above is equivalent to a version using a modulus of continuity.  In that version there is a universal
modulus of continuity $\sigma$ such that
\begin{equation}
   \frac{|\Lambda(w) \cap B_{\tilde{t}}|}{|B_{\tilde{t}}|} \geq \sigma(\tilde{t})
\label{eq:OtherVersion}
\end{equation}
for any $\tilde{t}$ implies a uniform convergence of the density of $\Lambda(w)$ to $1/2$ once $B_{\tilde{t}}$
is scaled to $B_1.$  (Here we mean uniformly among all appropriate $w$'s.)
\end{remark}

\begin{pf}
We start by assuming that we have a $t$ such that Equation\refeqn{bigonce}holds, and by rescaling if necessary,
we can assume that $t = r_0.$  Next, by arguing exactly as in the last theorem, by assuming that $r_0$ is
sufficiently small, and by defining $s_0 := \sqrt{r_0},$ we can assume without loss of generality that
\begin{equation}
\myavinttwo{B_{s_0}} \left| a^{ij}(x) - \delta^{ij} \right| \; dx
\label{eq:atbt}
\end{equation}
is as small as we like.  Now we will follow the argument given for Theorem 4.5 in \cite{B} very closely.

Applying our measure stability theorem on the ball $B_{s_0}$ we have the existence of a function
$u$ which satisfies:
\BVPcbsn{\Delta u = \chisub{\{u > 0\}}}{u \equiv w}{newudef}
and so that 
\begin{equation}
|\{\Lambda(u) \Delta \Lambda(w)\} \cap B_{r_0}|
\label{eq:damnsmall}
\end{equation}
is small enough to guarantee that
\begin{equation}
   \frac{|\Lambda(u) \cap B_{r_0}|}{|B_{r_0}|} \geq \frac{\epsilon}{2} \;,
\label{eq:uLbigonce}
\end{equation}
and therefore
\begin{equation}
   m.d.(\Lambda(u) \cap B_{r_0}) \geq C(n) r_0 \epsilon \;.
\label{eq:uLmdbigonce}
\end{equation}
Now if $r_0$ is sufficiently small, then by the $C^{1,\alpha}$ regularity theorem (Theorem\refthm{C1Al}\!\!) 
we conclude that $\partial \Lambda(u)$ is $C^{1,\alpha}$ in
an $r_0^2$ neighborhood of the origin. Furthermore, if we rotate coordinates so that
$FB(u) = \{ (x', x_n) \; | \; x_n = f(x') \},$
then we have the following bound (in $B_{r_0^2}$):
  \begin{equation}
  ||f||_{_{C^{1,\alpha}}} \leq \frac{C(n)}{r_0} \;.
  \label{eq:fSmoothBd}
  \end{equation}
On the other hand, because of this bound, there exists a $\gamma < 1$ such
that if $\rho_0 := \gamma r_0 < r_0,$ then 
  \begin{equation}
    \frac{ | \Lambda(u) \cap B_{\rho_0} | }{| B_{\rho_0} | }
    \; > \; \frac{1 - \epsilon}{2} \;.
    \label{eq:BrideOfBigLamb}
  \end{equation}
Now by once again requiring $r_0$ to be sufficiently
small, we get
  \begin{equation}
    \frac{ | \Lambda(w) \cap B_{\rho_0} | }{| B_{\rho_0} | }
    \; > \; \frac{1}{2} - \epsilon \;.
    \label{eq:HoundOfBigLamb}
  \end{equation}
(So you may note that here our requirement on the size of $r_0$ will be much smaller than it was before; we need it small
both because of the hypotheses within Caffarelli's regularity theorems and because of the need to shrink the $L^p$ norm of
$|a^{ij} - \delta^{ij}|$ in order to use our measure stability theorem.)

Now since $\frac{1}{2} - \epsilon$ is strictly
greater than $\epsilon,$  
we can rescale $B_{\rho_0}$ to a ball with a radius
\textit{close to} $r_0,$ and then repeat.  Since we have
a little margin for error in our rescaling, after
we repeat this process enough times we will have
a small enough radius (which we call $\tau r_0$),
to ensure that for all $r \leq \tau r_0$ we have
$$\frac{ | \Lambda(w) \cap B_r |}{| B_r |} 
  \; > \; \frac{1}{2} - \epsilon \;.$$
\end{pf}

\begin{corollary}[The Set of Regular Points Is Open]  \label{TSoRPIO}
If we take $w$ as above, then the set of regular points of $FB(w)$ is an open subset of $FB(w).$
\end{corollary}

\noindent
The proof of this corollary is identical to the proof of Corollary 4.8 in \cite{B} except that in place of using Theorem 4.5 of
\cite{B} we use Theorem\refthm{CAMS}\!\!.

\begin{corollary}[Persistent Regularity]  \label{PS}
Let $A^{ij}$ be a constant symmetric matrix with eigenvalues in $[\lambda, \Lambda].$ Let $w$ satisfy
$w \geq 0,$
$$A^{ij} D_{ij} w = \chisub{ \{ w > 0 \} } ,$$
and assume that $FB(w) \cap B_{3/4}$ is $C^{1,\alpha}.$  If $a^{ij}(x) \in \text{VMO} \cap L^{\infty}(B_1),$
and
$$||a^{ij} - A^{ij}||_{L^{q}(B_1)}$$
is sufficiently small, then the solution, $w_a,$ to the obstacle problem:
$$w_{a} \geq 0 \;, \ \ \ \ a^{ij}(x) D_{ij} w_{a}(x) = \chisub{ \{ w_a > 0 \} }(x) , \ \ \ \ w_{a} = w \ \text{on} \ \partial B_1$$
has a regular free boundary in $B_{1/2}.$  (In other words the density of $\Lambda(w_a)$ is equal to $1/2$ at every
$x \in FB(w_a) \cap B_{1/2}.$)
\end{corollary}

\begin{pf}
We start by observing that by Theorem\refthm{RegPts}there will be a neighborhood of $FB(w) \cap B_{5/8}$ where $w(x)$
will satisfy:
\begin{equation}
\gamma^{-1} \cdot \text{dist}(x, \Lambda(w))^2 \leq w(x) \leq \gamma \cdot \text{dist}(x, \Lambda(w))^2 \;,
\label{eq:goodnbhd}
\end{equation}
for a constant $\gamma > 0.$  By the same theorem, the size of this neighborhood will be bounded from below by a
constant, $\beta,$ which depends only on the $C^{1,\alpha}$ norm of $FB(w) \cap B_{3/4}.$  In other words,
Equation\refeqn{goodnbhd}will hold for all $x \in \Lambda(w)_{\beta} \cap B_{5/8}.$  On the other hand, in
$\Lambda(w)_{\beta}^c \; \cap \tclosure{B_{5/8}}$ the function $w$ will attain a positive minimum.
By applying Corollary\refthm{UnifStab}to guarantee that $$||w- w_a||_{L^{\infty}(B_1)}$$
is as small as we like, we 
can ensure that $w_a > 0$ in $\Lambda(w)_{\beta}^c \; \cap \tclosure{B_{5/8}},$ and so
$FB(w_a) \subset \Lambda(w)_{\beta}.$  By using Theorem\refthm{NonDeg}applied to $w_a,$ we can even guarantee that
\begin{equation}
FB(w_a) \cap B_{5/8} \subset FB(w)_{\beta} \;.
\label{eq:gooddistbeta}
\end{equation}

Now fix $0 < \tilde{\epsilon} <\!\!< \epsilon \leq 1/100.$
We choose $\tilde{\beta} < \beta$ based on the $C^{1,\alpha}$ norm of $FB(w)$ to ensure that for any $x_0 \in FB(w) \cap B_{5/8}$
and any $r \in (0, \tilde{\beta}]$ we have the inequality:
\begin{equation}
\left| \frac{|B_{r}(x_0) \cap \Lambda(w)|}{|B_{r}(x_0)|} - \frac{1}{2} \right| < \epsilon \;.
\label{eq:goodrat}
\end{equation}
Arguing exactly as above and shrinking $||a^{ij} - A^{ij}||_{L^q(B_1)}$ if necessary, we can now guarantee that
\begin{equation}
FB(w_a) \cap B_{5/8} \subset FB(w)_{(\tilde{\epsilon}\tilde{\beta})} \;.
\label{eq:gooddisttildebeta}
\end{equation}

Now pick a $y_0 \in FB(w_a) \cap B_{1/2},$ and let $x_0$ be a point in $FB(w)$ which minimizes distance to $y_0.$
Observe that
$$B_{\tilde{\beta}}(x_0) \subset B_{\tilde{\beta}(1 + \tilde{\epsilon})}(y_0)$$
and by $C^{1,\alpha}$ regularity we know
$$|FB(w)_{(\tilde{\epsilon}\tilde{\beta})} \cap B_{\tilde{\beta}}(x_0)| \leq C(n) \tilde{\epsilon} \tilde{\beta}^n \;.$$
\includegraphics{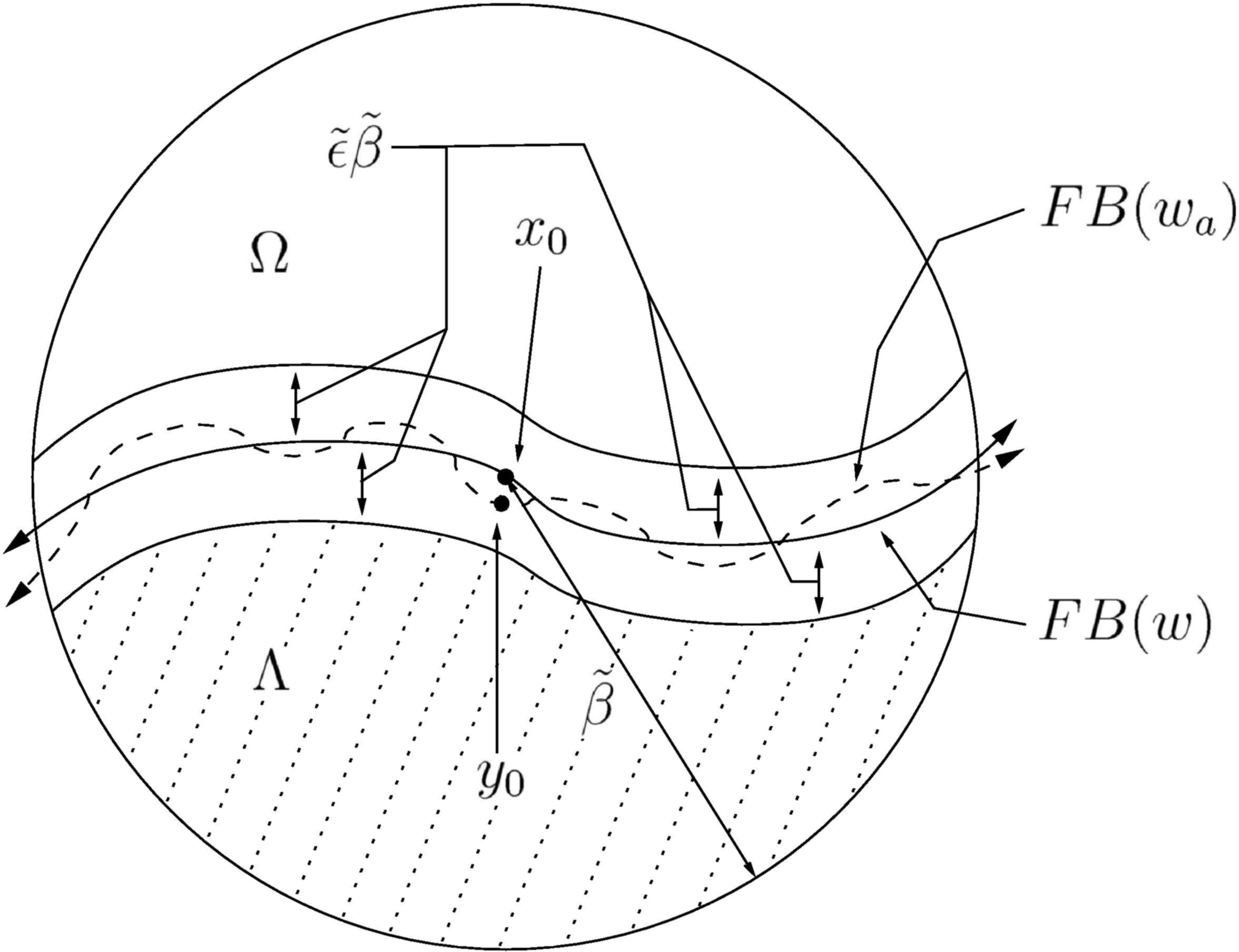} \newline \newline \noindent
In the figure, the region which is dotted represents the region
$$\Lambda(w) \cap B_{\tilde{\beta}}(x_0) \setminus FB(w)_{\tilde{\epsilon}\tilde{\beta}}$$
and it is necessarily a subset of $\Lambda(w_a) \cap B_{\tilde{\beta}(1 + \tilde{\epsilon})}(y_0).$
Using these observations along with Equations\refeqn{gooddisttildebeta}and\refeqn{goodrat}we estimate:
\begin{alignat*}{1}
\displaystyle{\frac{|\Lambda(w_a) \cap B_{\tilde{\beta}(1 + \tilde{\epsilon})}(y_0)|}{|B_{\tilde{\beta}(1 + \tilde{\epsilon})}(y_0)|}} \;
   &\displaystyle{\geq \frac{|\Lambda(w) \cap B_{\tilde{\beta}}(x_0)| 
                             - C(n)\tilde{\epsilon} \tilde{\beta}^n}{|B_{\tilde{\beta}(1 + \tilde{\epsilon})}(x_0)|} } \\
   &\displaystyle{= \frac{1}{(1 + \tilde{\epsilon})^n} \cdot \frac{|\Lambda(w) \cap B_{\tilde{\beta}}(x_0)| 
                             - C(n)\tilde{\epsilon} \tilde{\beta}^n}{|B_{\tilde{\beta}}(x_0)|} } \\
   &\displaystyle{\geq \frac{1}{(1 + \tilde{\epsilon})^n} \cdot \left[ \left(\frac{1}{2} - \epsilon\right) - C(n)\tilde{\epsilon} \right] } \\
   &\displaystyle{\geq 1/4} \;,
\end{alignat*}
as long as we choose our constants sufficiently small.  Now by shrinking the value of $\tilde{\beta}$ (if necessary) to be 
sure that $\tilde{\beta}(1 + \tilde{\epsilon})$ is less than the $r_0$ given in Theorem\refthm{CAMS}we can be sure that
$y_0$ is a regular point of $FB(w_a).$
\end{pf}


\newsec{An Important Counter-Example}{ICE}
Now we will give an example of a solution to an obstacle problem of the type we have been studying above which has more
than one blowup limit at the origin.  The first step will be to construct a convenient discontinuous function in
$\text{VMO} \cap L^{\infty}(B_1).$  

We define the function $f_k(x)$ by letting $f_k(x):= \gamma_k(|x|)$ where $\gamma_k(r)$ is defined by 
\begin{equation}
\gamma_k(r) := \left\{
    \begin{array}{cl}
        \displaystyle{2} \ \ \ \ &\displaystyle{\text{for} \ r \geq \omega_k}  \\
        \ & \ \\
        \displaystyle{\frac{5 + \cos( \pi \log | \log r | )}{2}} \ \ \ \ &\displaystyle{\text{for} \ r  <  \omega_k}
    \end{array} \right.
\label{eq:rkdefin}
\end{equation}
and $\omega_k := \exp( - \exp( 2k+1 ) )\;.$  (Note $\omega_k \downarrow 0,$ as $k \rightarrow \infty.$)
Now we observe the following properties:
\begin{enumerate}
   \item $$2 \leq f_k \leq 3 \ \text{in} \ B_1,$$
   \item $$\text{for any} \ q < \infty, \ \ \lim_{k \rightarrow \infty} ||f_k - 2||_{L^q(B_1)} = 0 \;, \ \ \text{and}$$ 
   \item $$\lim_{r \downarrow 0} r\gamma_k^{\prime}(r) = 0 \;.$$
\end{enumerate}
It now follows from a Theorem of Bramanti (using the first and third propery above) that $f_k(x) \in \text{VMO}(B_1).$
(See Theorem\refthm{RVMOBram}\!\!.)
Since we were not able to find this theorem published elsewhere we will include the proof in an appendix.  (This proof is
due to Bramanti and is found in his PhD dissertation: Commutators of singular integrals and parabolic equations with VMO
coefficients. Ph.D. Thesis, University of Milano, Italy, 1993. \cite{Br})

Now we define $a^{ij,k}(x) := f_k(x) \delta^{ij},$ and $p_{\beta}(x) := \frac{1}{4}((x_n - \beta)_{+})^{2}.$
Observe that $p_{\beta}$ solves the obstacle problem:
$$2\Delta w = \chisub{ \{ w > 0 \} } \;,$$
and $FB(p_{\beta}) = \{ x_n = \beta \} \;.$  Now for $-1/10 \leq \beta \leq 1/10$ and $k \in \N,$ we let $w_{\beta, k}$
denote the solution to the obstacle problem:
$$w \geq 0, \ \ \ a^{ij,k}(x) D_{ij} w = \chisub{ \{ w > 0 \} } \ \text{in} \ B_1, \ \ \ w(x) = p_{\beta}(x) \ \text{on} \ \partial B_1 \;.$$
Now we observe that
$$2\Delta(p_{\beta} - w_{\beta, k}) = 
     \chisub{ \{ p_{\beta} > 0 \} } - (2\delta^{ij} - a^{ij,k})D_{ij} w_{\beta, k} - \chisub{ \{ w_{\beta, k} > 0 \} } \;,$$
and so, since $||D_{ij} w_{\beta, k}||_{L^{2p}(B_{1/2})} \leq C$ which is independent of $k,$ and
since $(2\delta^{ij} - a^{ij,k}(x))$ vanishes outside of $B_{\omega_k},$ we have
\begin{alignat*}{1}
    ||2\Delta(p_{\beta} - w_{\beta, k})||_{L^{p}(B_1)}
        &\leq ||\chisub{ \{ p_{\beta} > 0 \} }  - \chisub{ \{ w_{\beta, k} > 0 \} }||_{L^{p}(B_1)} \\
         & \ \ \ \ + ||(2\delta^{ij} - a^{ij,k})D_{ij} w_{\beta, k}||_{L^{p}(B_1)}  \\ 
        &=|\{ \Lambda(p_{\beta}) \Delta \Lambda(w_{\beta, k}) \cap B_1|^{1/p} \\
         & \ \ \ \ + ||(2\delta^{ij} - a^{ij,k})D_{ij} w_{\beta, k}||_{L^{p}(B_{\omega_k})} \\
        &\leq|\{ \Lambda(p_{\beta}) \Delta \Lambda(w_{\beta, k}) \cap B_1|^{1/p} \\
         & \ \ \ \ + ||D_{ij} w_{\beta, k})||_{L^{2p}(B_{1/2})} \cdot ||(2\delta^{ij} - a^{ij,k})||_{L^{2p}(B_{\omega_k})} \\
        &\leq|\{ \Lambda(p_{\beta}) \Delta \Lambda(w_{\beta, k}) \cap B_1|^{1/p} \\
         & \ \ \ \ + C||(2\delta^{ij} - a^{ij,k})||_{L^{2p}(B_{\omega_k})}
\end{alignat*}
The first term can be made as small as we like by letting $k$ be very large and then by using measure stability, and the 
second term can be made as small as we like by letting $k$ be very large and by observing that
$$||(2\delta^{ij} - a^{ij,k})||_{L^{2p}(B_{\omega_k})} \leq |B_{\omega_k}|^{1/2p}.$$
Since $(p_{\beta} - w_{\beta, k}) \in W^{2,p}(B_1) \cap W^{1,p}_0(B_1)$
we can use Lemma 9.17 of \cite{GT} to guarantee that
$||p_{\beta} - w_{\beta, k}||_{W^{2,p}(B_1)}$ is as small as we like for any $p < \infty$ and therefore by the Sobolev
embedding
\begin{equation}
     ||p_{\beta} - w_{\beta, k}||_{L^{\infty}(B_1)} \ \  \text{is as small as we like.}
\label{eq:linfest}
\end{equation}
(We have not hesitated to increase $k.$)

By using Corollary\refthm{PS}\!\!, if we let $k$ be sufficiently large, then every $x \in FB(w_{\beta, k}) \cap B_{1/2}$
is a regular free boundary point.  (Here we mean ``regular'' in the sense of definition\refthm{RSFBP}\!\!.)  By applying Theorem\refthm{NonDeg}to
the $w_{\beta, k}$ and by using\refeqn{linfest}\!\!, we can assert that for all $\beta \in [-1/10,1/10],$ as long as
$k$ is sufficiently large, $$FB(w_{\beta, k}) \cap B_{1/2} \subset \{ \beta - 1/100 < x_n \}.$$  Now by observing that
$p_\beta(x) \geq (1/4) \cdot (1/100)^2$ in the set $\{ \beta + 1/100 \leq x_n \},$ we can use\refeqn{linfest}again, to guarantee
that $$FB(w_{\beta, k}) \cap B_{1/2} \subset \{ x_n < \beta + 1/100 \}$$ as long as $k$ is sufficiently large.  Thus
$$FB(w_{\beta, k}) \cap B_{1/2} \subset \{ \beta - 1/100 < x_n < \beta + 1/100 \}.$$
Arguing similarly, we can assert that
$$FB(p_\beta) \cap B_{1/2} \subset \{ (FB(w_{\beta, k}) \cap B_{1/2}) \}_{1/100} \;,$$
where for any $S \subset \R^n,$ we let $S_r$ denote the $r-neighborhood$ of the set $S.$

Now we claim that there exists a $\beta_0$ such that $0 \in FB(w_{\beta_0, k}),$ and since our function $f_k(x)$ oscillates
between $2$ and $3$ infinitely many times as we zoom in toward the origin, we can apply Theorem\refthm{EBL}to guarantee the
existence of different blowup limits.  To establish the claim, start by letting $\beta_0$ be the infimum of the $\beta$ such that
$0 \in \Lambda(w_{\beta,k}).$  It follows from Theorem\refthm{WCP}that for $-1/10 \leq \beta < \beta_0$ we have
$w_{\beta,k}(0) > 0,$ and for $\beta_0 \leq \beta \leq 1/10$ we have $w_{\beta,k}(0) = 0.$  (It follows from
Theorem\refthm{WCP}that if the boundary data is converging uniformly, then the solutions are converging uniformly, so
in particular, $w_{\beta_0,k}(0)=0.$)
Now if $0 \in FB(w_{\beta_0,k})$ then we are done.  On the other hand, since the $FB(w_{\beta_0,k})$ is a closed set,
if $0$ does not belong to $FB(w_{\beta_0,k}),$ then there exists an $\tilde{r} > 0$ such that
$\closure{B_{\tilde{r}}} \; \subset \Lambda(w_{\beta_0,k}).$

Now define $\beta_n := \beta_0 - 20^{-n}$ and observe that by Theorem\refthm{WCP}we know that
$w_n(x) := w_{\beta_n,k}(x)$ will converge uniformly
to $w_0(x) := w_{\beta_0,k},$ which is equal to zero on all of $B_{\tilde{r}}.$  On the other hand, since
$w_n(0) > 0$ we can apply Theorem\refthm{NonDeg}at the origin to conclude that
$$\sup_{B_{\tilde{r}}} w_n(x) \geq C\tilde{r}^2 \;,$$
and this fact makes uniform convergence to zero impossible, so the claim has been established.  We summarize this work in the
following theorem.

\begin{theorem}[Counter-Example]   \label{CountEx}
There exists $a^{ij} \in \text{VMO}(B_1)$ which satisfies Equation\refeqn{UniformEllip}with $\lambda = 2$ and
$\Lambda = 3,$ there exists a nonnegative solution $w(x)$ to Equation\refeqn{BasicProb}with this matrix $a^{ij},$
and there exists $\{r_n\} \downarrow 0$ such that 
$$\lim_{n \rightarrow \infty} w_{r_{2n + 1}}(x) = \frac{1}{4}((x_n - \beta)_{+})^{2}$$
and
$$\lim_{n \rightarrow \infty} w_{r_{2n}}(x) = \frac{1}{9}\tau(((x_n - \beta)_{+})^{2})$$
where $\tau$ is a rotation, and where the limits have the same convergence as in Theorem\refthm{EBL}(and where as
usual we let $w_{\epsilon}(x) := \epsilon^{-2}w(\epsilon x)$).
\end{theorem}

\newsec{Appendix}{Appendix}

This theorem and its proof are due to Bramanti.  On the other hand, since we have altered the exposition slightly, if
there are any mistakes, then they are due to the authors of this paper and not due to Bramanti.

\begin{theorem}[Radial VMO]   \label{RVMOBram}
Let $f:(0,R]\rightarrow \R$, $f\in C^{1}(0,R],$ and assume the following:
\begin{enumerate}
\item $f\in L^{2}(0,R) $

\item $xf(x) ^{2}\rightarrow 0$ for $x\rightarrow 0^{+}$

\item $xf^{\prime }(x) \rightarrow 0$ for $x\rightarrow 0^{+}$

\item $\frac{1}{r}\int_{0}^{r}x(f(r) - f(x)) f^{\prime }(x) dx\rightarrow 0$
         as $r\downarrow 0.$

\end{enumerate}
(Note that if $f$ is bounded, then it is enough to assume 3).

Let $u:B_{R}(0) \subset \R^{n}\rightarrow \R$
$$u(x) =f(|x|) .$$
Then $u\in \text{VMO}(B_{R}(0)) $.
\end{theorem}

\noindent

Before we prove the theorem, we prove the following lemma. We will consider the case $n=1.$
The general case can be handled similarly by radial change of variables. In this case $u$ is an even function on
$\left[-r, r\right]$.

\begin{lemma}
If $f$ and $u$ satisfy the same hypotheses as in the last theorem, then
$$\psi \left( r\right) := \frac{1}{2r}\int_{-r}^{r}\left\vert u\left(
x\right) -u_{\left( -r,r\right) }\right\vert ^{2}dx\rightarrow 0\text{ as }%
r\rightarrow 0.$$
\end{lemma}

\begin{proof}
By integration by parts
\begin{alignat*}{1}
   f_{\left( 0,r\right) } &= \frac{1}{r}\int_{0}^{r}f\left( x\right) dx \\
& = \frac{1}{r}(xf(x)) \rule[-.12in]{.01in}{.34in}_{\; 0}^{\; r} - \frac{1}{r} \int_{0}^{r}xf^{\prime }(x) dx \\
& = f(r) - \frac{1}{r} \int_{0}^{r}xf^{\prime}(x) dx \;.
\end{alignat*}
From this equation along with our third assumption it follows that
$$f_{(0,r)} - f(r) \rightarrow 0 \ \ \text{as} \ \ r \downarrow 0.$$
Thus
\begin{alignat*}{1}
\psi(r)
&= \frac{1}{r} \int_{0}^{r} |f(x)-f_{(0,r)}|^2dx \\
& \leq \frac{2}{r} \int_{0}^{r} |f(x)-f(r)|^2dx + o(1) \ \text{as} \ r \downarrow 0\\
& \leq \frac{2}{r} \left(x(f(x)-f(r))^2\right)\rule[-.12in]{.01in}{.34in}_{\; 0}^{\; r}  \\
& \ \ \ \ - \frac{4}{r} \int_{0}^{r} x(f(x)-f(r))f'(x)dx+ o(1) \rightarrow 0
\end{alignat*}
which proves the lemma.
\end{proof}

\noindent
Now we define
$$\eta _{2,u}^{\ast }(r) =
\sup_{x\in B_{R}(0), \; 0<\sigma <r}\frac{1}{|B_{\sigma}(x) \cap B_{R}(0)|}
\int_{B_{\sigma}(x) \cap B_{R}(0)}\left|u(x) -u_{B_{\sigma}(x) \cap B_{R}(0) }\right|^{2}dx \;.$$
To prove the theorem, it suffices to show that
$$\eta _{2,u}^{\ast }\left( r\right) \downarrow 0 \ \text{as} \ r \rightarrow 0.$$

\begin{proof}
(for $n=1,R=1$)
We will write $(a,b)^{\ast}:=(a,b) \cap (-1,1)$.
To bound $\eta _{2,u}^{\ast}(r) $ (for $n=1,R=1),$ let
$$ \Psi (x_{0},\varepsilon) =
\frac{1}{|(x_{0}-\varepsilon ,x_{0}+\varepsilon)^{\ast}|}
      \int_{(x_{0}-\varepsilon ,x_{0}+\varepsilon) ^{\ast}} |u(x) - u_{(x_{0}-\varepsilon,x_{0}+\varepsilon) ^{\ast}}|^{2}dx.$$
Without loss of generality, we can assume $x_{0}\geq 0,$ and we observe that
$$\int_{a}^{b}\left| u(x) -u_{(a,b) }\right|^{2}dx
    =\min_{\lambda \in \mathbb{R}}\int_{a}^{b}\left|u(x)-\lambda \right|^{2}dx.$$
We split the proof into two cases:
\begin{enumerate}
    \item $0 \leq x_{0} < 2\varepsilon.$ We can take $\varepsilon < \frac{1}{3}.$
             Then $(x_{0}-\varepsilon,x_{0}+\varepsilon) \subset (-3\varepsilon ,3\varepsilon) \subset (-1,1)$ and so
             \begin{alignat*}{1}
                   \Psi (x_{0},\varepsilon)
                         &\leq \frac{1}{2\varepsilon}
                              \int_{x_{0}-\varepsilon}^{x_{0}+\varepsilon}\left|u(x) - u_{(-3\varepsilon, 3\varepsilon) }\right|^{2}dx\\
                         &\leq 3\frac{1}{6\varepsilon}
                              \int_{-3\varepsilon }^{3\varepsilon }\left|u(x) - u_{(-3\varepsilon, 3\varepsilon)}\right|^{2}dx\\
                         &\leq 3\psi (3\varepsilon)\\
                         &\rightarrow 0
              \end{alignat*}
              as $\varepsilon \rightarrow 0,$ by the lemma above.
     \item $2\varepsilon \leq x_{0}<1.$ Then
              $(x_{0}-\varepsilon, x_{0}+\varepsilon)^{\ast}\subset [\varepsilon,1]$ and
              $$\Psi (x_{0},\varepsilon) \leq \omega _{\varepsilon}^{2}(2\varepsilon)$$
              where
              $$\omega _{\varepsilon}(h) := \max_{\left| x-y\right| \leq h; \; x,y\in [\varepsilon,1] }\left| f(x) - f(y) \right|.$$
              Since $f\in C^{1}[\varepsilon,1],$
              $$\omega _{\varepsilon }(h) \leq h\cdot \max_{x \in [\varepsilon ,1] }\left|f^{\prime}(x) \right|.$$
              Now, if $f^{\prime}$ is bounded on $(0,1],$ then we have
              $\omega _{\varepsilon}(2\varepsilon) \leq c\varepsilon,$ which already gives us what we need.  Otherwise:
              $$\omega_{\varepsilon}(2\varepsilon)
                        = \tilde{h} \left|f^{\prime}(\xi _{\varepsilon}) \right|$$
              for some $\xi _{\varepsilon}\in [\varepsilon,1], \ \text{and some} \ \tilde{h} \in (0, 2\varepsilon].$
              Thus
              \begin{eqnarray*}
                   \omega_{\varepsilon}(2\varepsilon)
                        &\leq
                        &2\varepsilon \left|f^{\prime}(\xi _{\varepsilon}) \right| \
                                          \text{and so} \\
                   \omega _{\varepsilon}(2\varepsilon)
                        &\leq
                        &2\xi _{\varepsilon}\left|f^{\prime}(\xi _{\varepsilon})\right|.
               \end{eqnarray*}
               Also, $\xi _{\varepsilon}\rightarrow 0$ as $\varepsilon \rightarrow 0,$
               since $f^{\prime}$ is bounded away from the origin. Finally we note that our third assumption
               implies that 
               $2\xi_{\varepsilon}\left|f^{\prime}(\xi _{\varepsilon})\right| \rightarrow 0.$
               In any case, $\omega _{\varepsilon}(2\varepsilon) \rightarrow 0$ for $\varepsilon \rightarrow 0,$
               and this fact implies
               $$\sup_{x_{0}\in \left[ 0,1\right] }\Psi \left( x_{0},\varepsilon \right)
                    \rightarrow 0\text{ for }\varepsilon \rightarrow 0,$$
\end{enumerate}
Since $\Psi \rightarrow 0$ in both cases, we can conclude that $u\in \text{VMO}(B_{1}(0)).$
\end{proof}


\section*{Acknowledgments}
The first author would like to thank Chuck Moore for a useful conversation, Diego Maldonado for pointing out some useful
references about VMO, and Luis Caffarelli for bringing his attention to the work of Chiarenza, Frasca, and Longo.  The
second author would like to thank Dian Palagachev for helping with a reference.  Both authors would like to thank Hrant
Hakobyan for a useful observation, and Virginia Naibo for a careful reading of the manuscript along with many excellent
suggestions for improvements to the exposition. Both authors are very grateful to Marco Bramanti for explaining his result
which gives a sufficient condition to guarantee that a radial function belongs to VMO, and for allowing them to include it here.
Finally, both authors would like to thank the anonymous referee for his/her careful reading and excellent suggestions.



\end{document}